\documentclass[10pt]{amsart}
\usepackage[utf8]{inputenc}
\usepackage[T1]{fontenc}
\usepackage{textcomp}
\usepackage[english]{babel}
\usepackage{hyperref}
\usepackage{url}
\usepackage{amsmath,amssymb,amsthm}
\usepackage{mathrsfs}
\usepackage{mathtools}
\usepackage[all,cmtip]{xy}
\usepackage{manfnt}
\usepackage{pdfpages}

\usepackage{lmodern}

\theoremstyle{definition}
\newtheorem{definition}{Definition}[section]
\newtheorem*{definition*}{Definition}
\theoremstyle{plain}
\newtheorem{proposition}[definition]{Proposition}
\newtheorem{lemma}[definition]{Lemma}
\newtheorem{theorem}[definition]{Theorem}
\newtheorem*{theorem*}{Theorem}
\newtheorem{corollary}[definition]{Corollary}

\theoremstyle{remark}

\newtheorem{remark}[definition]{Remark}

\numberwithin{equation}{section}

\renewcommand\epsilon{\varepsilon}
\renewcommand\phi{\varphi}
\newcommand\RR{\mathbb{R}}
\newcommand\CC{\mathbb{C}}
\newcommand\ZZ{\mathbb{Z}}
\newcommand\QQ{\mathbb{Q}}
\newcommand\NN{\mathbb{N}}
\newcommand\kk{\mathbf{k}}

\newcommand\Alg{\mathbf{Alg}}

\newcommand\Art{\mathbf{Art}}

\newcommand\ProArt{\mathbf{ProArt}}

\newcommand\Set{\mathbf{Set}}

\newcommand\Grp{\mathbf{Grp}}

\newcommand\Ohat{\widehat{\mathcal{O}}}

\DeclareMathOperator\id{id}
\DeclareMathOperator\Ker{Ker}

\DeclareMathOperator*\End{End}
\DeclareMathOperator*\Hom{Hom}

\DeclareMathOperator\Ad{Ad}

\DeclareMathOperator\Gr{Gr}

\DeclareMathOperator\Def{Def}
\DeclareMathOperator\MC{MC}

\DeclareMathOperator\Cone{Cone}

\DeclareMathOperator\Sym{Sym}
\DeclareMathOperator\Free{\mathscr{F}}

\DeclareMathOperator\FreeGCoAlg{\Free^{\mathnormal{c}}}

\DeclareMathOperator\MHD{MHD}

\newcommand\TW{\mathrm{TW}}

\hypersetup{
pdfauthor = {Louis-Clément LEFÈVRE},
pdftitle = {Mixed Hodge structures and representations of
fundamental groups of algebraic varieties},
pdfsubject = {Mixed Hodge structures and representations of
fundamental groups of algebraic varieties},
pdfkeywords = {Complex algebraic geometry,
Hodge theory, Fundamental groups, Representation varieties, Formal deformation
theory, L-infinity algebras},
linktocpage = {true},
}

\title[Mixed Hodge structures and representations]{Mixed Hodge structures and representations of
fundamental groups of algebraic varieties}
\author{Louis-Clément Lefèvre}
\email{louisclement.lefevre@univ-grenoble-alpes.fr}
\address{Univ.\ Grenoble Alpes, CNRS, Institut Fourier, F-38000 Grenoble, France}

\begin{document}

\begin{abstract}
Given a complex variety $X$, a linear algebraic group $G$ and a representation
$\rho$ of the fundamental group $\pi_1(X,x)$ into $G$, we develop a framework for
constructing a functorial mixed Hodge structure on the formal local ring of the
representation variety of $\pi_1(X,x)$ into $G$ at $\rho$ using mixed Hodge
diagrams and methods of $L_\infty$ algebras. We apply it in two geometric
situations: either when $X$ is compact Kähler and $\rho$ is the monodromy of a
variation of Hodge structure, or when $X$ is smooth quasi-projective and $\rho$
has finite image.
\end{abstract}

\maketitle

\tableofcontents
\section{Introduction}

Classical Hodge theory provides the cohomology groups of compact Kähler
manifolds with additional structure. This structure has been vastly used as a
tool to study and restrict the possible topology of these manifolds. It was
generalized by Deligne \cite{DeligneII} into the notion of mixed Hodge structure for complex
algebraic varieties. Since then, mixed Hodge structures have been constructed on
many other topological invariants. In particular Morgan \cite{Morgan} and Hain
\cite{Hain} constructed mixed Hodge structures on the rational homotopy groups
of such varieties.

In this vein, in this article we construct mixed Hodge structures on invariants
associated to linear representations of the fundamental group. Let $X$ be either
a compact Kähler manifold or a smooth complex quasi-projective algebraic
variety, with a base point $x$. In both cases the fundamental group $\pi_1(X,x)$
is finitely presentable. Let $G$ be a linear algebraic group over some field
$\kk\subset\RR$ or $\CC$. The set of group morphisms from $\pi_1(X,x)$ into
$G(\kk$) is the set of points of an affine scheme of finite type over $\kk$ that
we denote by $\Hom(\pi_1(X,x),G)$ and call the \emph{representation variety}.
Given a representation $\rho$ of $\pi_1(X,x)$, seen as a $\kk$-point of the
representation variety, we denote by $\Ohat_\rho$ the formal local ring at
$\rho$. This ring pro-represents the functor of deformations of $\rho$ and its
structure tells about the topology of $X$.

The study of $\Ohat_\rho$ was done first by Goldman and Millson
\cite{GoldmanMillson} when $X$ is a compact Kähler manifold: they show that it
has a quadratic presentation. This was later reviewed by Eyssidieux-Simpson
\cite{EyssidieuxSimpson} who constructed a functorial mixed Hodge structure on
$\Ohat_\rho$ when $\rho$ is the monodromy of a variation of Hodge structure.
In case where $X$ is non-compact and $\rho$ has finite image, Kapovich and
Millson \cite{KapovichMillson} constructed only quasi-homogeneous gradings on
$\Ohat_\rho$. They apply it to describe explicitly a class of finitely
presented groups that cannot be fundamental groups of smooth algebraic
varieties. We improve their construction to a mixed Hodge structure:

\begin{theorem*}
There is a mixed Hodge structure on $\Ohat_\rho$ in the following situations:
\begin{enumerate}
\item Either $X$ is compact Kähler and $\rho$ is the monodromy of a polarized
variation of Hodge structure.
\item Either $X$ is smooth quasi-projective and $\rho$ has finite image.
\end{enumerate}
This mixed Hodge structure is defined over the field $\kk$ is $\rho$ is (this
includes the case $\kk=\CC$, where one works with complex mixed Hodge
structures). It is functorial in $X,x,\rho$. The weight zero part is the formal
local ring defining the orbit by conjugation of $\rho$ inside
$\Hom(\pi_1(X,x),G)$.
\end{theorem*}

By the way we present a framework which unifies the previous constructions.
Namely in their work, Eyssidieux-Simpson make strong use of the special
properties of analysis on compact Kähler manifolds and nothing can be directly
generalized if $X$ is not compact. On the other hand, Kapovich-Millson make
strong use of the theory of minimal models of Morgan and they only get a grading
and not a mixed Hodge structure on $\Ohat_\rho$; furthermore this construction
is not functorial.  We propose here a different approach in which in all cases
the mixed Hodge structure comes directly from the $H^0$ of an appropriate mixed
Hodge complex whose construction depends on the geometric situation, but the
functoriality and the dependence on the base point is very explicit. Let us
explain.

The theory of Goldman-Millson was one of the starting point of the use of
differential graded (DG) Lie algebras in deformation theory. Let $L$ be the DG
Lie algebra of differential forms over $X$ with values in the adjoint bundle of
the flat principal $G$-bundle induced by the monodromy of $\rho$. To $L$ one can
associate a deformation functor $\Def_L$ which is related to the functor of
formal deformations of $\rho$. We review this in
sections~\ref{section:deformation-DG-Lie} and~\ref{section:Goldman-Millson}. The
classical theorems known to Goldman and Millson are that $\Def_L$ is invariant
under quasi-isomorphisms on $L$ and is pro-representable if $H^0(L)=0$.  In our
compact case, $L$ has the special property of being formal (i.e.\
quasi-isomorphic to its cohomology) and the deformation functor simplifies
much.

Since then the theory of formal deformations via DG Lie algebras has received
many contributions by many people, including Hinich, Kontsevich, Manetti,
Pridham, Lurie. See the survey \cite{Toen}. It is then very interesting to
review the classical theory used by Goldman-Millson and Kapovich-Millson in the
light of the much more modern tools developed: derived deformation theory and
$L_\infty$ algebras. In particular this theory furnishes a canonical algebra
pro-representing the deformation functor of $L$ (in case $H^0(L)=0$) and
invariant under quasi-isomorphisms. This algebra is obtained by dualizing a
coalgebra which is the $H^0$ of a DG coalgebra $\mathscr{C}(L)$ known as the
\emph{bar construction} or \emph{Chevalley-Eilenberg complex}. We explain this
in section~\ref{section:L-inf}. In our non-compact case $L$ is not formal and
using this construction is the first part of our unifying framework.

The second part of the unification is that in all our geometric situations, $L$
has a structure of mixed Hodge complex (a notion introduced by Deligne in order
to construct mixed Hodge structures on the cohomology of a complex) compatible
with the multiplicative structure. We call this a \emph{mixed Hodge diagram of
Lie algebras}. In the compact case this is almost an immediate consequence of
the work of Deligne-Zucker \cite{Zucker} constructing Hodge structures on the
cohomology with coefficients in a variation of Hodge structure. In the
non-compact case the ideas are already contained in the work of Kapovich-Millson
but they use the method of Morgan and we re-write them with the more powerful
and more functorial method of Navarro Aznar \cite{Navarro}.
From the data of a mixed Hodge diagram of Lie algebras only, we would like to
get the mixed Hodge structure on $\Ohat_\rho$. Hence we delay the construction
of these diagrams to the sections~\ref{section:geo-compact} (compact case)
and~\ref{section:geo-non-compact} (non-compact case).

However in the theory of Goldman-Millson $L$ controls the deformation theory of
$\rho$ only up to conjugation. If we want to eliminate conjugations, we have to
work with $L$ together with its augmentation $\epsilon_x$ to $\mathfrak{g}$ the
Lie algebra of $G$ given by evaluating degree zero forms at $x$. To this data is
associated by Eyssidieux-Simpson an \emph{augmented deformation functor}
$\Def_{L,\epsilon_x}$ which is a slight variation of $\Def_L$ and is isomorphic
to the functor of formal deformation of $\rho$, pro-represented by
$\Ohat_\rho$. Thus any algebra pro-representing $\Def_{L,\epsilon_x}$ will be
canonically isomorphic to $\Ohat_\rho$. It also coincides with the deformation
functor associated to the DG Lie algebra $\Ker(\epsilon_x)$ but this may not be
anymore a mixed Hodge diagram. In the category of mixed Hodge complexes it is
better to work with the mapping cone, but the mapping cone of an augmented DG
Lie algebra may not be anymore a DG Lie algebra.  Despite our efforts, it is not
possible to find an object that would be at the same time DG Lie, mixed Hodge
complex, and with $H^0=0$.

We find a solution to this issue by working with $L_\infty$ algebras. Briefly,
these are DG Lie algebras \emph{up to homotopy} that have the same properties as
DG Lie algebras at the level of cohomology but are less rigid to construct (for
example, any DG vector space with a quasi-isomorphism to a DG Lie algebra
inherits a structure of $L_\infty$ algebra). They also have an associated
deformation functor that extends the classical one of DG Lie algebras. In an
article of Fiorenza-Manetti \cite{FiorenzaManetti} the mapping cone of a
morphism between DG Lie algebras is shown to carry a canonical $L_\infty$
algebra structure and the associated deformation functor is studied. We explain
this in section~\ref{subsection:Linfinifty-mapping-cone}.  The crucial claim for
us is that for augmented DG Lie algebras the deformation functor of the cone
coincides with the augmented deformation functor introduced independently by
Eyssidieux-Simpson.

The plan of our approach is now clear. Start with the algebraic data of a mixed
Hodge diagram of Lie algebras $L$ together with an augmentation $\epsilon$ to a Lie
algebra $\mathfrak{g}$ (we review the tools we need from Hodge theory in
sections~\ref{section:MHS} and~\ref{section:MHC}). We show first in
section~\ref{section:MHD-L-infinity} that on the mapping cone $C$ of $\epsilon$
the structures of mixed Hodge complex and of $L_\infty$ algebra are
compatible. We call the resulting object a \emph{mixed Hodge diagram of
$L_\infty$ algebras}. Then $C$ has $H^0=0$ so its deformation functor is
pro-representable, again by the algebra dual to $H^0(\mathscr{C})$ where the
functor $\mathscr{C}$ is naturally extended from DG Lie algebras to $L_\infty$
algebras. We then show in section~\ref{section:bar-MHD} that $\mathscr{C}(C)$ is
a \emph{mixed Hodge diagram of coalgebras}. This step is very similar to Hain's
\emph{bar construction on mixed Hodge complexes}. And then its $H^0$ and the dual
algebra, which is canonically isomorphic to $\Ohat_\rho$ by its
pro-representability property, have mixed Hodge structures. The functoriality of
this construction is clear and the dependence on the base point is explicit. In
section~\ref{section:MHS-desctiption} we extract some description of the mixed
Hodge structure we constructed by algebraic methods, especially we describe it
at the level of the cotangent space and we describe the weight zero part.

The reason why we are restricted to representations with finite image when $X$
is non-compact is that until now we are unable to construct a mixed Hodge
diagram in more general cases, for example when $\rho$ is the monodromy of a
variation of Hodge structure. We strongly believe it is possible to deal with
this case with our methods. The case of singular varieties may also be
accessible. Combined with a further study of group theory, this could lead to
new examples of finitely presentable groups not isomorphic to fundamental groups
of smooth algebraic varieties.

\subsection*{Acknowledgment} This is part of my Ph.D.\ thesis. I am very
grateful to my supervisor Philippe Eyssidieux and to the thesis
reviewers Nero Budur and Carlos Simpson. I also thank the
project ANR Hodgefun.

\section{Deformation functor of a DG Lie algebra}
\label{section:deformation-DG-Lie}

We start by recalling the classical framework for deformation theory: functors
on local Artin algebras and DG Lie algebras. We fix a field $\kk$ of
characteristic zero.

\begin{definition}
\label{definition:local-Artin-algebra}
A \emph{local Artin algebra} over $\kk$ is an algebra $A$ over $\kk$, local with
maximal ideal $\mathfrak{m}_A$, with residue field $A/\mathfrak{m}_A=\kk$, and
finite-dimensional over $\kk$.  These form a category $\Art_\kk$, where the
morphisms are morphisms of algebras over $\kk$ that are required to preserve the
maximal ideals.
\end{definition}

This definition implies that $\mathfrak{m}_A$ is a nilpotent ideal and so $A$ is
local complete. By \emph{deformation functors} we will mean certain functors
from $\Art_\kk$ to the category of sets $\Set$.

\begin{definition}
\label{definition:pro-Artin-algebra}
A \emph{pro-Artin algebra} over $\kk$ is a complete local algebra $R$ over $\kk$
such that all quotients $R/(\mathfrak{m}_R)^n$ are local Artin algebras (so $R$
is a projective limit of local Artin algebras).  These form a category
$\ProArt_\kk$ where morphisms have to preserve the maximal ideal.
\end{definition}

The most basic examples are the formal power series algebras
$\kk[[ X_1,\dots,X_r]]$. The category $\Art_\kk$ is a full subcategory of
$\ProArt_\kk$, which is itself a full subcategory of the category of complete
local algebras with residue field $\kk$. Namely in all these three categories,
morphisms are morphisms of algebras that preserve the maximal ideal.
Among all complete local algebras, those that we call pro-Artin are
characterized by the fact that their cotangent space
$\mathfrak{m}_R/(\mathfrak{m}_R)^2$ is finite-dimensional over $\kk$,
or equivalently they are characterized by the Noetherian property.

Any such algebra defines a deformation functor:

\begin{definition}
A functor $F:\Art_\kk \rightarrow \Set$ is said to be \emph{pro-representable}
if $F$ is isomorphic to
\begin{equation}
A\longmapsto \Hom_{\ProArt}(R, A)
\end{equation}
for some pro-Artin algebra $R$ over $\kk$.
\end{definition}

By the pro-Yoneda lemma (\cite[\S~2]{Schlessinger}, \cite[3.1]{GoldmanMillson}),
such an algebra $R$ is unique up to a unique isomorphism.

We then refer to the lecture notes \cite[\S~V]{ManettiLectures} for DG Lie algebras. Such a
DG Lie algebra $L$ has a grading $L=\bigoplus_{n\in\ZZ} L^n$, a differential
\begin{equation}
d:L^n \longrightarrow L^{n+1}
\end{equation}
and an anti-symmetric Lie bracket
\begin{equation}
[-,-]:L^i \otimes L^j \longrightarrow L^{i+j}
\end{equation}
such that $d$ is a derivation for the Lie bracket and the bracket satisfies the
graded Jacobi identity. Its set of \emph{Maurer-Cartan elements} is
\begin{equation}
\MC(L) := \left\{ x\in L^1 \ \left|\  d(x)+\frac{1}{2}[x,x]=0 \right. \right\} .
\end{equation}
If $A$ is a local Artin algebra with maximal ideal $\mathfrak{m}_A$ then
$L\otimes\mathfrak{m}_A$ is a \emph{nilpotent} DG Lie algebra with bracket
\begin{equation}
[u\otimes a, v\otimes b] := [u,v]\otimes ab,\quad u,v\in L,\ a,b\in\mathfrak{m}_A
\end{equation}
and differential
\begin{equation}
d(u\otimes a) := d(u)\otimes a .
\end{equation}
On $L^0 \otimes\mathfrak{m}_A$, which is a nilpotent Lie algebra, the
Baker-Campbell-Hausdorff formula defines a group structure that we denote by
$(\exp(L^0\otimes\mathfrak{m}_A), *)$, whose elements are denoted by $e^\alpha$ for
$\alpha\in L^0\otimes\mathfrak{m}_A$. This groups acts on
$\MC(L\otimes\mathfrak{m}_A)$ by \emph{gauge transformations}.

\begin{definition}
\label{definition:DG-Lie-deformation-functor}
The \emph{deformation functor} of a DG Lie algebra $L$ over $\kk$ is the functor
\begin{equation}
\begin{array}{rcl}
\Def_L : \Art_\kk & \longrightarrow & \Set \\
A & \longmapsto & \frac{\MC(L\otimes\mathfrak{m}_A)}{\exp(L^0
                  \otimes\mathfrak{m}_A)} .
\end{array}
\end{equation}
If a deformation functor $F$ is isomorphic to some $\Def_L$ we say that $L$
\emph{controls} the deformation problem given by $F$.
\end{definition}

The main theorem of deformation theory is:

\begin{theorem}[{\cite[V.52]{ManettiLectures}}]
\label{theorem:main-theorem-deformation-theory}
Let $\phi:L\stackrel{\approx}{\longrightarrow} L'$ be a quasi-isomorphism
between DG Lie algebras. Then the induced morphism
\begin{equation}
\Def_L \stackrel{\simeq}{\longrightarrow} \Def_{L'}
\end{equation}
is an isomorphism.
\end{theorem}

When reviewing the theory of Goldman-Millson, Eyssidieux and Simpson introduce a
small variation of the classical deformation functor for augmented DG Lie
algebras.

\begin{definition}[{\cite[\S~2.1.1]{EyssidieuxSimpson}}]
Let $\epsilon:L\rightarrow \mathfrak{g}$ be an augmentation of the DG Lie
algebra $L$ to the Lie algebra $\mathfrak{g}$ (that one can see as a DG Lie
algebra concentrated in degree zero). The
\emph{augmented deformation functor} is the functor
\begin{equation}
\Def_{L,\epsilon}:\Art_\kk \longrightarrow \Set
\end{equation}
defined on $A \in \Art_\kk$ as the quotient of the set
\begin{equation}
\label{equation:objects-deformation-functor-augmented}
\left\{ (x,
  e^{\alpha})\in (L^1\otimes\mathfrak{m}_A)\times \exp(\mathfrak{g}\otimes\mathfrak{m}_A) \ \left| \
    d(x)+\frac{1}{2}[x,x] = 0 \right. \right\}
\end{equation}
by the action of $\exp(L^0 \otimes \mathfrak{m}_A)$ given by
\begin{equation}
\label{equation:morphisms-deformation-functor-augmented}
e^{\lambda}.(x,e^{\alpha}) := \big(e^{\lambda}.x,\ e^\alpha * e^{-\epsilon(\lambda)}\big) .
\end{equation}
\end{definition}

The $*$ on the right-hand side
of~\eqref{equation:morphisms-deformation-functor-augmented} is the product given
by the Baker-Campbell-Hausdorff formula. For this group law
$e^{-\epsilon(\lambda)}$ is the inverse of $e^{\epsilon(\lambda)}$, thus this is
really a left action of the group $\exp(L^0\otimes\mathfrak{m}_A)$.

\begin{remark}
This is not exactly the same action as in \cite{EyssidieuxSimpson}, but we make
this choice so as to be compatible with the construction developed in
section~\ref{subsection:Linfinifty-mapping-cone}.
\end{remark}

The fundamental Theorem~\ref{theorem:main-theorem-deformation-theory} has a
small variation for the augmented deformation functor.

\begin{lemma}[{See \cite[2.7]{EyssidieuxSimpson}}]
\label{lemma:quasi-isomorphism-deformation-functor-augmented}
Let ${\phi:L\stackrel{\approx}{\longrightarrow} L'}$ be a quasi-isomorphism of
augmented DG Lie algebras commuting with the augmentations:
\begin{equation}
\xymatrix{ L \ar[rd]_{\epsilon} \ar[rr]_{\phi}^{\approx} & & L' \ar[ld]^{\epsilon'} \\ & \mathfrak{g}.
&  }
\end{equation}
Then $\phi$ induces an isomorphism of deformation functors
\begin{equation}
\Def_{L,\epsilon} \stackrel{\simeq}{\longrightarrow} \Def_{L', \epsilon'}.
\end{equation}
\end{lemma}

\section{The theory of Goldman and Millson}
\label{section:Goldman-Millson}
Then we give a brief account of \cite{GoldmanMillson}.
We fix a linear algebraic group $G$ over $\kk=\QQ$, $\RR$ or $\CC$. We think of
$G$ as a representable functor
\begin{equation}
G:\Alg_\kk\longrightarrow \Grp
\end{equation}
from the category of algebras over $\kk$ to the category of groups. We are
interested in studying a group $\Gamma$ by looking at its representations into
$G(\kk)$, and actually into all $G(A)$ for varying algebras $A$ over $\kk$.

\begin{definition}[{\cite{LubotzkyMagid}}]
\label{theorem:variety-of-representations}
For any finitely generated group $\Gamma$ the functor
\begin{equation}
\begin{array}{rcl} \Hom(\Gamma, G):\Alg_\kk & \longrightarrow & \Set \\ A & \longmapsto &
\Hom_{\Grp}(\Gamma, G(A)) \end{array}
\end{equation}
is represented by an affine scheme of finite type over $\kk$. We denote it again by
$\Hom(\Gamma, G)$ (we think of it as a scheme structure on the set $\Hom(\Gamma,
G(\kk))$) and call it the \emph{representation variety} of $\Gamma$ into $G$. If
$\rho$ is a representation of $\Gamma$ into $G(\kk)$, seen as a point over
$\kk$ of $\Hom(\Gamma, G)$, we denote by $\Ohat_\rho$ the completion of the
local ring of $\Hom(\Gamma, G)$ at $\rho$.
\end{definition}

The local ring $\Ohat_\rho$ will be one of our main objects of study. It is a
pro-Artin algebra whose corresponding pro-representable functor on local Artin
algebras over~$\kk$
\begin{equation}
\label{equation:pro-representable-functor-complete-local-ring}
\begin{array}{rcl} \Hom(\Ohat_\rho,-): \Art_\kk & \longrightarrow & \Set \\ A & \longmapsto &
\Hom_{\ProArt}(\widehat{\mathcal{O}}_\rho, A)
\end{array}
\end{equation} is canonically isomorphic to the functor of \emph{formal deformations of $\rho$}
\begin{equation}
\label{equation:functor-formal-deformations-representation} \Def_\rho : A \longmapsto \big\{
\tilde{\rho}:\Gamma\rightarrow G(A) \ \big|\ \tilde{\rho}=\rho \bmod\mathfrak{m}_A \big\}.
\end{equation}
Define $G^0$ to be the functor
\begin{equation}
\begin{array}{rcl} G^0:\Art_\kk & \longrightarrow & \Grp \\ A & \longrightarrow & \big\{ g \in G(A)\
\big|\ g=1_{G(\kk)} \bmod \mathfrak{m}_A \big\}
\end{array}
\end{equation}
then $G^0(A)$ acts on $\Def_\rho(A)$ by conjugation, functorially in $A$. We
simply denote by $\Def_\rho/G^0$ the quotient functor from $\Art_\kk$ to
$\Set$. If $\mathfrak{g}$ is the Lie algebra of $G$ then $G^0(A)$ has
$\mathfrak{g}\otimes \mathfrak{m}_A$ as Lie algebra and actually one can
construct it as
\begin{equation}
\label{equation:G0-exp} G^0(A) = \exp(\mathfrak{g}\otimes\mathfrak{m}_A) .
\end{equation}

Let now $X$ be a manifold whose fundamental group is finitely presentable. This
hypothesis encompasses both the compact Kähler manifolds and the smooth
algebraic varieties over $\CC$. Let $x$ be a base point of $X$. We are
interested in the deformation theory of representations of the fundamental group
$\pi_1(X,x)$. Over both fields $\RR$ and $\CC$, $G(\kk)$ is a Lie group. To a
representation
\begin{equation}
\rho:\pi_1(X,x) \longrightarrow G(\kk)
\end{equation}
corresponds a flat principal $G(\kk)$-bundle
\begin{equation}
P := \widetilde{X} \times_{\pi_1(X,x)} G(\kk)
\end{equation}
(where $\widetilde{X}$ is the universal cover of $X$, seen as a principal
$\pi_1(X,x)$-bundle). The adjoint bundle is
\begin{equation}
\Ad(P) := P\times_{G(\kk)}\mathfrak{g} = \widetilde{X} \times_{\pi_1(X,x)}
\mathfrak{g} 
\end{equation}
where $\pi_1(X,x)$ acts on $\mathfrak{g}$ via $\Ad\circ \rho$. Thus $\Ad(P)$ is a
local system of Lie algebras with fiber at $x$ canonically identified to
$\mathfrak{g}$. Let now
\begin{equation}
L := \mathscr{E}^\bullet(X, \Ad(P))
\end{equation}
be the DG Lie algebra of $\mathcal{C}^\infty$ differential forms with values in
$\Ad(P)$. Locally an element of $L$ is given by a  sum $\sum \alpha_i\otimes
u_i$ where $\alpha_i$ is a differential form and $u_i$ is an element of
$\mathfrak{g}$. The differential is given by
\begin{equation}
\label{equation:differential-L}
d(\alpha\otimes u) := d(\alpha) \otimes u
\end{equation}
and the Lie bracket is given by
\begin{equation}
\label{equation;Lie-bracket-L}
[\alpha\otimes u, \beta\otimes v] := (\alpha \wedge\beta) \otimes [u,v] .
\end{equation}
Part of the main theorem of Goldman and Millson can be stated as follows:

\begin{theorem}[{\cite{GoldmanMillson}}]
\label{theorem:Goldman-Millson-main}
There is a canonical isomorphism of deformation functors
\begin{equation}
\Def_L \stackrel{\simeq}{\longrightarrow} \Def_\rho/G^0 .
\end{equation}
\end{theorem}

Roughly, this is a deformed version of the usual correspondence between
representations of the fundamental group and principal bundles with a flat
connection: a representation of $\pi_1(X,x)$ modulo conjugation corresponds to a
flat principal bundle modulo automorphisms, or equivalently to a fixed principal
bundle with a flat connection modulo the action of the automorphism group on
flat connections.

However we are interested in the functor $\Def_\rho$ and not his quotient by
conjugation by $G^0$, because $\Def_\rho$ is pro-represented by the pro-Artin
algebra $\Ohat_\rho$.  For this we introduce the augmentation
\begin{equation}
\label{equation:augmentation-L}
\begin{array}{rcl}
\epsilon_x:L & \longrightarrow & \mathfrak{g} \\
\omega \in L^0 & \longmapsto & \omega(x) \\
\omega\in L^{>0} & \longmapsto & 0 .
\end{array}
\end{equation}
Then the correspondence of Goldman-Millson becomes the following:

\begin{theorem}[{\cite{GoldmanMillson}, \cite[2.7]{EyssidieuxSimpson}}]
\label{theorem:Goldman-Millson-main-augmented}
The isomorphism of Theorem~\ref{theorem:Goldman-Millson-main} induces a
canonical isomorphism of deformation functors
\begin{equation}
\Def_{L,\epsilon_x} \stackrel{\simeq}{\longrightarrow} \Def_\rho .
\end{equation}
\end{theorem}

When $X$ is a compact Kähler manifold and $\rho$ has image contained in a
compact subgroup of $G(\RR)$, or $\rho$ is the monodromy of a polarized
variation of Hodge structure over $X$, then $L$ is formal (i.e.\
quasi-isomorphic to a DG Lie algebra with zero differential) and the functor
$\Def_{L,\epsilon_x}$ simplifies much thanks to
Theorem~\ref{theorem:main-theorem-deformation-theory} (or its augmented variant,
Lemma~\ref{lemma:quasi-isomorphism-deformation-functor-augmented}). This allows
Goldman and Millson to give a description of $\Ohat_\rho$ via its property that
it pro-represents the deformation functor $\Def_{L,\epsilon}$ and one can
construct quite easily such an object when $L$ has $d=0$. However to deal with
more general geometric situations where $L$ is not formal we must study more
carefully how to pro-represent $\Def_{L,\epsilon_x}$.

\section{\texorpdfstring{$L_\infty$}{L-infinity} algebras and pro-representability}
\label{section:L-inf}
We will need to work with $L_\infty$ algebras for two well-distinct reasons. The
first one is for giving a pro-representability theorem for the deformation
functor of DG Lie algebras. The second one is for explaining how the augmented
deformation functor of $\epsilon:L\rightarrow\mathfrak{g}$ comes from a
$L_\infty$ algebra structure on the mapping cone of $\epsilon$. Our reference is
again the lectures notes \cite[\S~VIII--IX]{ManettiLectures}. Through the
sections~\ref{section:MHD-L-infinity}--\ref{section:bar-MHD} we will show that these constructions have
good compatibilities with Hodge theory.

Briefly, $L_\infty$ algebras are weakened versions of DG Lie algebras in which
the Jacobi identity only holds \emph{up to homotopy}. Such a $L_\infty$ algebra
is given by a graded vector space $L$ with a sequence of anti-symmetric linear maps
\begin{equation}
\label{equation:L-infinity-maps-l}
\ell_r : L^{\wedge r} \longrightarrow L \quad (r\geq 1)
\end{equation}
where $\ell_r$ has degree $2-r$.  The precise definition is best stated in terms
of codifferential on a cofree conilpotent coalgebra and we will need this point
of view.

So let us say some words on coalgebras first. In the following we will work with
DG coalgebras $X$, equipped with a comultiplication
\begin{equation}
\Delta:X \longrightarrow X \otimes X
\end{equation}
that is always assumed to be coassociative and cocomutative, and with a
differential $d$ on $X$ that is a coderivation for $\Delta$; it is called the
\emph{codifferential}.

\begin{definition}
On a graded coalgebra $X$, there is a \emph{canonical filtration} indexed by
$\NN$ given by
\begin{equation}
X_n := \Ker\big( \Delta^n:X\rightarrow X^{\otimes (n+1)} \big)
\end{equation}
where $\Delta^n$ is the iterated comultiplication. The coalgebra is said to be
\emph{conilpotent} (terminology of \cite[1.2.4]{LodayValette}; these are called
\emph{locally conilpotent} in \cite{ManettiLectures}) if the canonical filtration is exhaustive, i.e.
\begin{equation}
X = \bigcup_{n\geq 0} X_n .
\end{equation}
\end{definition}

The dual of a coalgebra, in our sense, is a commutative algebra without
unit. Conilpotent coalgebras are dual to maximal ideals of complete local
algebras and conilpotent coalgebras with finite canonical filtration are dual to
maximal ideals of local algebras with nilpotent maximal ideal. So if $X$ is such
a coalgebra then $\kk\oplus X^*$ is an algebra with unit and with maximal ideal
$X^*$.

\begin{lemma}
\label{lemma:conilpotent-coalgebras-dual-pro-artin}
If $X$ is a conilpotent coalgebra whose canonical filtration is by
finite-dimensional sub-coalgebras, then its dual $X^*$ is the maximal ideal of a
pro-Artin algebra and conversely the dual of a pro-Artin algebra is such a
conilpotent coalgebra.
\end{lemma}

\begin{proof}
For such an $X$, $R:=\kk\oplus X^*$ is a complete local algebra with maximal
ideal $X^*$ whose quotients are finite-dimensional, hence $R$ is
pro-Artin. Conversely if $R$ is pro-Artin then $R$ is a projective limit of
finite-dimensional algebras, which dualizes to an inductive limit of coalgebras,
defining a conilpotent coalgebra $X$ with filtration by finite-dimensional sub-coalgebras.
\end{proof}

\begin{definition}[{\cite[VIII.24]{ManettiLectures}}]
\label{definition:cofree-coalgebra}
Let $V$ be a graded vector space. The \emph{cofree (conilpotent) coalgebra} on
$V$ is the graded vector space given by the reduced symmetric algebra
\begin{equation}
\FreeGCoAlg(V) := \Sym^{+}(V) = \bigoplus_{r\geq 1} V^{\odot r}
\end{equation}
(we denote by $\odot$ the symmetric product) endowed with the comultiplication
\begin{multline}
\label{equation:deconcatenation-coproduct} \Delta(v_1 \odot\cdots
\odot v_r) := \\ \sum_{p=0}^r \sum_{\tau\in\mathfrak{S}(p,r-p)}
\epsilon(\tau;v_1,\dots,v_r)\,(v_{\tau(1)} \odot \cdots \odot
v_{\tau(p)}) \otimes (v_{\tau(p+1)} \odot \cdots \odot v_{\tau(r)}) .
\end{multline}
where $\mathfrak{S}(p,r-p)\subset\mathfrak{S}(r)$ is the set of permutations of
$\{1,\dots,r\}$ that are increasing on $\{1,\dots,p\}$ and on $\{p+1,\dots,r\}$.
\end{definition}

This construction is a right adjoint to the forgetful functor from graded
conilpotent coalgebras to graded vector spaces. The canonical filtration here is
simply the natural increasing filtration of the symmetric algebra:
\begin{equation}
\label{equation:canonical-filtration-cofree-coalgebra}
\FreeGCoAlg(V)_n := \bigoplus_{r=1}^n V^{\odot r}.
\end{equation}
A coderivation $Q$ on $V$ is
determined uniquely by its components (\cite[VIII.34]{ManettiLectures})
\begin{equation}
q_r : V^{\odot r} \longrightarrow V \quad (r\geq 1).
\end{equation}

For any DG vector space $V$ and integer $i\in\ZZ$ we define the \emph{shift}
$V[i]$ which has
\begin{equation}
V[i]^n := V^{n+i}
\end{equation}
and differential $d_{V[i]}=(-1)^i d_V$. We denote by $x[i]$ an element of $V[i]$
obtained by shifting the degree by $i$ from an element $x$ of $V$ (if $x$ has
degree $n$ in $V$ then $x[i]$ is the element $x$ with degree $n-i$ in $V[i]$,
since $V[i]^{n-i}=V^n$).

\begin{definition}
A \emph{$L_\infty$ algebra} is the data of a graded vector space $L$ and a
codifferential $Q$ on $\FreeGCoAlg(V[1])$. We denote by $\mathscr{C}(L)$ this coalgebra.
\end{definition}

Thus the structure of $L_\infty$ algebra on $L$ is determined either by the maps
\begin{equation}
q_r:(L[1])^{\odot r} \longrightarrow L[1]
\end{equation}
of degree $1$, either by the maps $\ell_r$ of~\eqref{equation:L-infinity-maps-l}
obtained as shifts of $q_r$ (since we take symmetric and exterior products in
the graded sense, there is a canonical isomorphism
$(L[1])^{\odot r} \simeq L^{\wedge r}[r]$), with all axioms encoded in the
relation $Q\circ Q=0$. DG Lie algebras correspond exactly to $L_\infty$ algebras
with $\ell_r=0$ for $r\geq 3$, then $\ell_1$ is the differential of $L$ and
$\ell_2$ is the Lie bracket. The construction $\mathscr{C}$ is known in the
literature under the names of \emph{Quillen $\mathscr{C}$ functor}, \emph{bar
construction} or (homological) \emph{Chevalley-Eilenberg complex}. We prefer the
term of bar construction and we call the canonical
filtration~\eqref{equation:canonical-filtration-cofree-coalgebra} of
$\mathscr{C}(L)$ the \emph{bar filtration}.

\begin{definition}
A (strong) \emph{morphism} between $L_\infty$ algebras $L,M$ is a morphism of
graded vector spaces $L\rightarrow M$ that commutes with all the operations
$\ell_r$. A \emph{$L_\infty$-morphism} is a morphism of graded coalgebras
$\mathscr{C}(L)\rightarrow\mathscr{C}(M)$. 
\end{definition}

Any morphism induces canonically an $L_\infty$-morphism, and a
$L_\infty$-morphism has a \emph{linear part} which is a morphism of DG vector
spaces (i.e.\ commuting with $\ell_1$, the differential).

\begin{theorem}[{\cite[IX.9]{ManettiLectures}}]
\label{theorem:L-infinity-quasi-isomorphism}
A $L_\infty$-morphism from $L$ to $M$ whose linear part is a quasi-isomorphism
induces a quasi-isomorphism of DG coalgebras from $\mathscr{C}(L)$ to
$\mathscr{C}(M)$. 
\end{theorem}

We call such morphisms \emph{$L_\infty$-quasi-isomorphisms}. The fundamental
theorem of deformation theory with $L_\infty$ algebras takes the following form.

\begin{theorem}[{\cite[\S~IX]{ManettiLectures}}]
To any $L_\infty$ algebra $L$ over $\kk$ is associated a deformation functor
\begin{equation}
\Def_L : \Art_\kk \longrightarrow \Set
\end{equation}
which coincides for DG Lie algebras with the usual deformation functor of
Definition~\ref{definition:DG-Lie-deformation-functor} and is invariant under
$L_\infty$-quasi-isomorphisms.
\end{theorem}

Then we state the pro-representability theorem we will use. This appears in
various places: it is stated for example in the survey by Toën \cite{Toen}, in
the article of Hinich \cite[\S~9.3]{Hinich}, and in the lecture notes of
Kontsevich \cite{Kontsevich}. But we never find it written exactly in the form
we want in the literature, especially because of the duality and the
finite-dimensionality issue, so we give here the simplest possible proof using
only the above theorems.

\begin{theorem}[Pro-representability]
\label{theorem:pro-representability-theorem}
Let $L$ be a $L_\infty$ algebra with $H^n(L)=0$ for $n\leq 0$ and $H^1(L)$
finite-dimensional. Then $H^0(\mathscr{C}(L))$ is a coalgebra satisfying the
hypothesis of Lemma~\ref{lemma:conilpotent-coalgebras-dual-pro-artin} and
\begin{equation}
R: = \kk \oplus H^0(\mathscr{C}(L))^*
\end{equation}
is a pro-Artin algebra that pro-represents the deformation functor of $L$ and is
invariant under $L_\infty$-quasi-isomorphisms.
\end{theorem}

\begin{proof}
First, Theorem~\ref{theorem:L-infinity-quasi-isomorphism} states that
$H^0(\mathscr{C}(L))$ is invariant under $L_\infty$-quasi-iso\-mor\-phisms, and so is
$R$. The $L_\infty$ algebra structure on $L$ transfers, up to
$L_\infty$-quasi-isomorphism, to a $L_\infty$ algebra structure on its
cohomology (\cite[\S~10.3]{LodayValette}). Thus we can assume that $L$ has
$L^n=0$ for $n\leq 0$ and $L^1$ finite-dimensional. The deformation functor of
$L$ is then simply given by the solutions of the Maurer-Cartan equation
\begin{equation}
A\in\Art_\kk \longmapsto \left\{x\in L^1\otimes\mathfrak{m}_A\ \left|\ 0 =
  \sum_{r\geq 1}\frac{1}{r!}\, \ell_r(x\wedge\cdots\wedge x) \right. \right\}
\end{equation}
(this sum is finite since $\mathfrak{m}_A$ is nilpotent, and when $\ell_r=0$ for
$r\geq 3$ this is the usual Maurer-Cartan equation for DG Lie algebras), without
any quotient. But by a simple calculation (\cite[VIII.27]{ManettiLectures}) this
functor is the same as
\begin{equation}
A\in\Art_\kk \longmapsto \Hom((\mathfrak{m}_A)^*, H^0(\mathscr{C}(L))) .
\end{equation}
Assuming that $L^1$ is finite-dimensional, the canonical filtration of
$\mathscr{C}(L)$ is by finite-dimensional sub-coalgebras and similarly for
$H^0(\mathscr{C}(L))$. Thus one can dualize as in
Lemma~\ref{lemma:conilpotent-coalgebras-dual-pro-artin} and the above functor is
isomorphic to 
\begin{equation}
A\in\Art_\kk \longmapsto \Hom(H^0(\mathscr{C}(L))^*, \mathfrak{m}_A)
\end{equation}
which is also $\Hom(R, A)$ with the $\Hom$ in the
category of local algebras.
\end{proof}

\begin{remark}
When $L$ is a DG Lie algebra and is formal, the above theorem is a simple
calculation combined with the invariance of $\Def_L$ under
quasi-isomorphisms. In the general case however, and even when $L$ is a DG Lie
algebra, this is really a theorem in derived deformation theory and for us it
improves much the classical theory (compare with \cite{GoldmanMillsonKuranishi},
\cite[\S~4]{Manetti}).
\end{remark}

\section{\texorpdfstring{$L_\infty$}{L-infinity} algebra structure on the mapping cone}
\label{subsection:Linfinifty-mapping-cone}

After the pro-representability theorem, the second main point of our approach is
the article of Fiorenza-Manetti \cite{FiorenzaManetti} showing that the mapping
cone of a morphism between DG Lie algebras has a canonical $L_\infty$ algebra
structure and describing explicitly the deformation functor.

\begin{definition}
\label{definition:mapping-cone}
Let $\epsilon:L\rightarrow M$ be a morphism between DG vector spaces.
The \emph{mapping cone} of $\epsilon$ is the DG
vector space $\Cone(\epsilon)$ with
\begin{equation} \Cone(\epsilon)^n := L^{n+1} \oplus M^n
\end{equation} and differential
\begin{equation} d^n_{\Cone(\epsilon)}(x,y) := \big(-d^{n+1}_L(x),\
d^n_M(y)-\epsilon(x)\big) .
\end{equation} It is sometimes more natural to work with the
\emph{desuspended mapping cone} $\Cone(\epsilon)[-1]$, which has
\begin{equation} \Cone(\epsilon)[-1]^n = L^n \oplus M^{n-1}
\end{equation} and differential
\begin{equation} d^n_{\Cone(\epsilon)[-1]}(x,y) = \big(d^{n}_L(x),\
\epsilon(x)-d^{n-1}_M(y)\big) .
\end{equation}
If $\epsilon$ is a morphism between DG Lie algebras, then
on the desuspended mapping cone one defines a \emph{naive bracket} which is a
bilinear map of degree zero given by
\begin{equation}
\label{equation:naive-bracket} \big[(x, u), (y,v)\big] := \left(
[x,y],\ \frac{1}{2}\,[u, \epsilon(y)] +
\frac{(-1)^{|x|}}{2}\,[\epsilon(x), v] \right)
\end{equation} for $x,y\in L$, $u,v\in M$.
\end{definition}

The naive bracket is anti-symmetric and the differential is a derivation for the
bracket, but in general it doesn't satisfy the Jacobi identity.
Let us sum up all what we need directly in one theorem.

\begin{theorem}[{\cite[Main Thm.]{FiorenzaManetti}}]
\label{theorem:main-Fiorenza-Manetti}
Let $\epsilon:L\rightarrow M$ be a morphism between DG Lie algebras and let $C$
be its desuspended mapping cone. Then $C$ has a $L_\infty$ algebra structure,
with $\ell_1$ the usual differential of the cone and $\ell_2$ the naive
bracket. 

The higher brackets of $C$, given as the components $q_r$ of a codifferential
$Q$ on $\mathscr{C}(C)$, are given as follows (we write $x,y,\dots$ for elements
of $L$ and $u,v,\dots$ for elements of $M$).  If $r+k\geq 3$ and $k\neq 1$ then
\begin{equation}
q_{r+k}(u_1\odot \cdots \odot u_r \otimes x_1 \odot \dots \odot
x_k) = 0 
\end{equation}
and for $r\geq 2$
\begin{multline}
\label{equation:description-higher-brackets-long} q_{r+1}(u_1 \odot
\cdots \odot u_r \otimes x) =\\ -(-1)^{\sum_{i=1}^r |u_i|}
\,\frac{B_r}{r!} \sum_{\tau\in \mathfrak{S}(r)}
\epsilon(\tau,u_1,\dots,u_r)\,
[u_{\tau(1)},[u_{\tau(2)},\dots,[u_{\tau(r)}, \epsilon(x)]\dots ]]
\end{multline}
where the $B_n$ are the Bernouilli numbers.

This $L_\infty$ algebra structure is functorial from the category of morphisms
between DG Lie algebras to the category of $L_\infty$ algebras with their strong
morphisms.

The associated deformation functor on $\Art_\kk$ is isomorphic to the quotient
functor of the set of objects (for $A\in\Art_\kk$)
\begin{equation}
\label{equation:objects-deformation-functor-Manetti}
\left\{ (x,
  e^{\alpha})\in (L^1\otimes\mathfrak{m}_A)\times
  \exp(M^0\otimes\mathfrak{m}_A) \ \left |\ d(x)+\frac{1}{2}[x,x] = 0,\
    e^{\alpha}*\epsilon(x)=0 \right. \right\}
\end{equation} by the relation
\begin{multline}
\label{equation:morphisms-deformation-functor-Manetti}
(x,e^\alpha) \sim (y,e^\beta) \quad \Longleftrightarrow \\ \exists (\lambda, \mu)
\in (L^0\otimes\mathfrak{m}_A)\times (M^{-1}\otimes\mathfrak{m}_A) ,\quad
  e^{\lambda}.x=y,\ e^{\beta}=e^{d(\mu)}*e^{\alpha}*e^{-\epsilon(\lambda)}.
\end{multline}
\end{theorem}

And our main remark for this section is:

\begin{lemma}
\label{lemma:main-augmented-deformation-Manetti-Eyssidieux-Simpson}
Applied to an augmentation $\epsilon:L\rightarrow\mathfrak{g}$ to a Lie algebra
$\mathfrak{g}$, the above deformation functor associated to the cone of
$\epsilon$ introduced by Fiorenza-Manetti coincides with the augmented
deformation functor defined by Eyssidieux-Simpson.
\end{lemma}

\begin{proof}
Apply the theorem with $M$ concentrated in degree zero: $\epsilon(x)=0$ and also
$\mu=0$ so $e^{d(\mu)}$ is the identity of
$\exp(L^0\otimes\mathfrak{m}_A)$. Then we see immediately that the set of
objects~\eqref{equation:objects-deformation-functor-augmented}
and~\eqref{equation:objects-deformation-functor-Manetti} are the same and the
equivalence relations
in~\eqref{equation:morphisms-deformation-functor-augmented}
and~\eqref{equation:morphisms-deformation-functor-Manetti} are also the same.
\end{proof}

\section{Mixed Hodge structures}
\label{section:MHS}

We now turn to Hodge theory and we adopt all the classical notations and
conventions from Deligne \cite{DeligneII} concerning filtrations and Hodge
structures. In particular we always denote by $W_\bullet$ an increasing
filtration and by $F^\bullet$ a decreasing filtration. They are all indexed by
$\ZZ$. Filtrations of vector spaces are assumed to be finite and filtrations of
DG vector spaces are assumed to be biregular (i.e.\ induce a finite filtration
on each component) unless we state explicitly that we work with inductive or
projective limits of such objects. We will work with a fixed field
$\kk\subset\RR$.

First recall the classical definitions of pure and mixed Hodge structures.

\begin{definition}
\label{definition:Hodge-structure}
A (pure) \emph{Hodge structure} of weight $k$ over the field $\kk$ is
the data of a finite-dimensional vector space $K$ over $\kk$ and a bigrading
over $\CC$
\begin{equation}
K_\CC := K \otimes_{\kk} \CC = \bigoplus_{p+q=k} K^{p,q}
\end{equation}
with the complex conjugation exchanging $K^{p,q}$ and $K^{q,p}$.
The \emph{Hodge filtration} is then the decreasing filtration $F^\bullet$ of
$K_\CC$ defined by
\begin{equation}
F^p K_\CC := \bigoplus_{p\geq p'} K^{p',q} .
\end{equation}
\end{definition}

\begin{definition}
\label{definition:mixed-Hodge-structure}
A \emph{mixed Hodge structure} over $\kk$ is the data of a
finite-dimensional vector space $K$ over $\kk$ with an increasing
filtration $W_{\bullet}$ called the \emph{weight filtration} and a
decreasing filtration $F^\bullet$ on $K_\CC$ such that for each $k\in\ZZ$ the graded piece
\begin{equation}
\Gr^{W}_k (K) := W_k(K) / W_{k-1}(K)
\end{equation}
with the induced filtration $F$ over $\CC$ forms a pure Hodge structure of
weight $k$.
\end{definition}

Pure and mixed Hodge structures are abelian categories equipped with a tensor
product and internal homs. In particular one can talk about various kinds of algebras (resp.\ of coalgebras) carrying a
compatible mixed Hodge structure: this is simply the data of a mixed Hodge
structure $K$ with a multiplication map $K\otimes K \rightarrow K$ (resp.\
a comultiplication $K\rightarrow K\otimes K$) that is a morphism of mixed Hodge
structures and satisfies the corresponding axioms of algebra or coalgebra.

We will also have to deal with inductive and projective limits of these: by
definition a pro-Artin algebra with a mixed Hodge structure is the data of a
pro-Artin algebra $R$ with a mixed Hodge structure on each
$R/(\mathfrak{m}_R)^n$ (which is finite-dimensional) such that the canonical morphisms
\begin{equation}
R/(\mathfrak{m}_R)^{n+1} \longrightarrow R/(\mathfrak{m}_R)^n
\end{equation}
are morphisms of mixed Hodge structures. Dually a conilpotent coalgebra with a
mixed Hodge structure is a conilpotent coalgebra $X$ with a mixed Hodge
structure on each term $X_n$ of the canonical filtration, compatible with the
inclusion $X_n\subset X_{n+1}$.

A mixed Hodge structure on $K$ defines (in several ways) a bigrading of $K_\CC$.

\begin{definition}[{\cite[3.4]{PetersSteenbrink}}]
\label{definition:Deligne-splitting}
Let $K$ be a mixed Hodge structure. Let $K^{p,q}$ be the
$(p,q)$-component of $\Gr^W_{p+q}(K)$. Define the subspace of $K_\CC$
\begin{equation}
I^{p,q} := F^p \cap W_{p+q} \cap \left( \overline{F}^q \cap W_{p+q} \cap \sum_{j\geq 2}
  \big(\overline{F}^{q-j+1} \cap W_{p+q-j} \big) \right) .
\end{equation}
Then this defines a bigrading
\begin{equation}
K_\CC = \bigoplus_{p,q} I^{p,q}
\end{equation}
such that the canonical projection $K_\CC \rightarrow \Gr^W_{p+q}(K_\CC)$
induces an isomorphism
\begin{equation}
I^{p,q} \simeq K^{p,q} .
\end{equation}
We call it the \emph{Deligne splitting}. It is functorial and
compatible with duals and tensor products.
\end{definition}

Actually we will use only the grading by weight $K_\CC=\bigoplus_k I^k$ with
\begin{equation}
I^k=\bigoplus_{p+q=k} I^{p,q}.
\end{equation}

\section{Mixed Hodge complexes}
\label{section:MHC}
For us the basic tool for constructing mixed Hodge structures will be Deligne's
notion of mixed Hodge complex, introduced in \cite[\S~8]{DeligneIII}. This
consists of several complexes related by filtered quasi-isomorphisms and axioms
implying that all cohomology groups carry mixed Hodge structures.

\begin{definition}
If $(K,W)$ and $(L,W)$ are filtered DG vector spaces, a \emph{filtered
quasi-isomorphism} is a morphism $f:K\rightarrow L$ of filtered DG vector spaces
that induces a quasi-isomorphism
\begin{equation}
\Gr^W_k(K) \stackrel{\approx}{\longrightarrow} \Gr^W_k(L)
\end{equation}
for all $k$. Since the filtrations are biregular, $f$ is in particular a
quasi-isomorphism.

If $(K,W,F)$ and $(L,W,F)$ are bifiltered DG vector spaces (equipped with two
filtrations, $W$ increasing and $F$ decreasing) a \emph{bifiltered
quasi-isomorphism}  is a morphism of bifiltered DG vector spaces $f:K\rightarrow
L$ that induces a quasi-isomorphism
\begin{equation}
\label{equation:bifiltered-quasi-ismorphism}
\Gr^W_k \Gr_F^p(K) \stackrel{\approx}{\longrightarrow} \Gr^W_k \Gr_F^p(L)
\end{equation}
for all $k,p$. By the Zassenhauss lemma (\cite[1.2.1]{DeligneII}) the two graded
pieces $\Gr^W_k \Gr_F^p$ and $\Gr_F^p \Gr_W^k$ are canonically isomorphic and
one can invert them in~\eqref{equation:bifiltered-quasi-ismorphism}. Again, $f$
is in particular a quasi-isomorphism.
\end{definition}

\begin{definition}[{\cite[8.1.5]{DeligneIII}}]
\label{definition:mixed-Hodge-complex}
A \emph{mixed Hodge complex} over $\kk$ is the data of a filtered bounded-below DG vector
space $(K_\kk,W_\bullet)$ over $\kk$ and a bifiltered bounded-below DG vector space
$(K_\CC,W_\bullet,F^\bullet)$, together with a chain of filtered quasi-iso\-mor\-phisms
\begin{equation}
(K_\kk, W)\otimes\CC \stackrel{\approx}{\longleftrightarrow} (K_\CC, W)
\end{equation}
satisfying the following axioms:
\begin{enumerate}
\item
\label{MHC-axiom-1}
For all $n\in\ZZ$, $H^n(K)$ is finite-dimensional.
\item 
\label{MHC-axiom-2}
For all $k\in\ZZ$, the differential of $\Gr_k^W (K_\CC)$ is
strictly compatible with the filtration $F$.
\item
\label{MHC-axiom-3}
For all $n\in\ZZ$ and all $k\in\ZZ$, the filtration $F$ induced on
$H^n(\Gr^W_k (K_{\CC}))$ and the form $H^n(\Gr^W_k (K_\kk))$ over $\kk$ are part
of a pure Hodge structure of weight $k+n$ over $\kk$ on $H^n(\Gr^W_k (K))$.
\end{enumerate}
\end{definition}

\begin{remark}
In the following, we will work with mixed Hodge complexes with a fixed chain of
quasi-isomorphisms from $K_\kk$ to $K_\CC$, via intermediate components over
$\CC$ that we denote by $K_i$ ($i$ is an index belonging to some indexing
category) carrying a filtration $W$. The morphisms $K_i\rightarrow K_j$ are
called the \emph{comparison quasi-isomorphisms}.  See
\cite{CiriciGuillenComplexes} and \cite{Cirici} for this point of view.
\end{remark}

\begin{theorem}[Deligne]
If $K$ is a mixed Hodge complex, then on each $H^n(K)$ the induced filtrations
$F$ and $W[n]$ (with $W[n]_k:=W_{k-n}$) define a mixed Hodge structure.
\end{theorem}

\begin{definition}
A \emph{morphism} $f$ between mixed Hodge complexes $K,L$ is the data of a
collection of morphisms $f_i:K_i\rightarrow L_i$ of filtered DG vector spaces,
and $f_\CC:K_\CC\rightarrow L_\CC$ of bifiltered DG vector spaces, commuting
with the comparison quasi-isomorphisms. It is said to be a
\emph{quasi-isomorphism} if $f_i$ is a filtered quasi-isomorphism and $f_\CC$ is
a bifiltered quasi-isomorphism.
\end{definition}

Let us describe three useful constructions on mixed Hodge complexes.

\begin{proposition}[{\cite[8.1.24]{DeligneIII}, \cite[3.20]{PetersSteenbrink}}]
\label{proposition:tensor-product-mixed-Hodge-complex}
The \emph{tensor product} of two mixed Hodge complexes $K,L$, defined
level-wise, is again a mixed Hodge complex. Its cohomology computes (via the
Künneth formula) the tensor product of the mixed Hodge structures on cohomology.
\end{proposition}

Similarly, the exterior and symmetric products of mixed Hodge complexes are
again mixed Hodge complexes, as well as the whole tensor, exterior and symmetric
algebras.

\begin{definition}
We call \emph{mixed Hodge diagram} a mixed Hodge complex $K$ having an
additional structure of algebra (commutative algebra, or Lie algebra, or
coalgebra).
\end{definition}

Equivalently this means that all components $K_i$ are algebras and the
comparison quasi-isomorphisms are morphisms of algebras.

\begin{remark}
Algebraically, this is just a little bit more than a mixed Hodge complex, for
example this is what Hain calls \emph{multiplicative mixed Hodge complex}
(\cite[\S~3.1]{Hain}). Geometrically however, the construction of a mixed Hodge
diagram of algebras whose cohomology computes the whole cohomology algebra of a
given variety together with its mixed Hodge structure is much more
difficult. This is the object of section~\ref{section:geo-non-compact}.
\end{remark}

\begin{proposition}
If $K$ is a mixed Hodge complex and $r\in\ZZ$, then the \emph{shift} $K[r]$,
which has in each component $K_i[r]^n := K_i^{n+r}$, equipped with the filtrations
$W[r]$ and $F$, is again a mixed Hodge complex.
\end{proposition}

This is the natural way of shifting a complex together with its mixed Hodge
structure on cohomology. We will use it mainly for $r=1$ or $r=-1$.

\begin{proposition}[{\cite[3.22]{PetersSteenbrink}}]
\label{proposition:mapping-cone-mixed-Hode-complex}
The \emph{mapping cone} of a morphism $f:K\rightarrow L$ between mixed Hodge
complexes  is again a mixed Hodge complex, with filtration $W$ in the component
$i$ given by
\begin{equation}
W_k \Cone(f)_i^n := W_{k-1} K_i^{n+1} \oplus W_k L_i^n
\end{equation}
and in the component over $\CC$ carrying also the filtration $F$
\begin{equation}
F^p \Cone(f)_\CC^n := F^p K_\CC^n \oplus F^p L_\CC^n .
\end{equation}
Combined with the preceding construction, the \emph{desuspended mapping cone}
is a mixed Hodge complex with
\begin{equation}
W_k (\Cone(f)[-1])_i^n := W_{k} K_i^{n} \oplus W_{k+1} L_i^{n-1}
\end{equation}
and
\begin{equation}
F^p (\Cone(f)[-1])_\CC^n := F^p K_\CC^{n} \oplus F^p L_\CC^{n-1} .
\end{equation}
The long exact sequence
\begin{equation}
\cdots \longrightarrow H^n(K) \longrightarrow H^n(L)
\longrightarrow H^{n}(\Cone(f)) \longrightarrow H^{n+1}(K)
\longrightarrow H^{n+1}(L) \longrightarrow \cdots
\end{equation}
(where one can replace $H^n(\Cone(f))$ by $H^{n+1}(\Cone(f)[-1])$) is a
long exact sequence of mixed Hodge structures.
\end{proposition}

\section{Mixed Hodge diagrams of \texorpdfstring{$L_\infty$}{L-infinity} algebras}
\label{section:MHD-L-infinity}

In this section we show the compatibility of the construction of
Fiorenza-Manetti with mixed Hodge complexes.  The main object of study that we
introduce will be the following.

\begin{definition}
\label{definition:MHD-L-infinity}
A \emph{mixed Hodge diagram of $L_\infty$ algebras} is the data of a mixed Hodge
complex $L$ where all components $L_i$ are $L_\infty$ algebras and such that the
operations $\ell_r$ in $r$ variables of degree $2-r$ ($r\geq 1$) are compatible
with the filtrations, in the sense that
\begin{equation}
\label{equation:MHD-L-infinity-W}
 \ell_r \left((W_{k_1} L_i^{n_1}) \wedge \cdots \wedge
  (W_{k_r} L_i^{n_r}) \right) \subset W_{k_1 + \cdots + k_r}
L_i^{n_1+\cdots+n_r+2-r}
\end{equation}
and
\begin{equation}
\ell_r \left((F^{p_1} L_\CC^{n_1}) \wedge \cdots \wedge 
  (F^{p_r} L_\CC^{n_r}) \right) \subset F^{p_1 + \cdots + p_r}
L_\CC^{n_1+\cdots+n_r+2-r} .
\end{equation}
Furthermore the comparison quasi-isomorphisms are required to be strong
morphisms of $L_\infty$ algebras. If in each component $\ell_r=0$ for $r\geq 3$
this is simply a \emph{mixed Hodge diagram of Lie algebras}. A \emph{morphism}
between mixed Hodge diagrams of $L_\infty$ algebras $L,M$ is given by a morphism
of mixed Hodge complexes commuting with the operations $\ell_r$.
\end{definition}

At the level of cohomology, this behaves exactly as one would expect from mixed
Hodge diagrams of Lie algebras.

\begin{proposition}
If $L$ is a mixed Hodge diagram of $L_\infty$ algebras (a fortiori, of Lie
algebras) then on cohomology the induced Lie bracket is a morphism of mixed
Hodge structures.
\end{proposition}

\begin{proof}
By passing to cohomology we forget all the operations $\ell_r$ for
$r\neq 2$ and $\ell_2$ becomes a Lie bracket $[-,-]$. The statement
that the Lie bracket respects $F$ is clear because the condition
\begin{equation}
\ell_2 \left((F^p L_\CC^n) \wedge (F^q L_\CC)^m \right) \subset F^{p+q} L_\CC^{n+m}
\end{equation}
directly induces on cohomology
\begin{equation}
\left[F^p H^n(L_\CC), F^q H^m(L_\CC) \right] \subset F^{p+q}
H^{n+m}(L_\CC) .
\end{equation}
For $W$, in any of the components $L_i$, one has to be more
careful. Take cohomology classes
\begin{equation}
[u] \in W[n]_k H^n(L_i), \quad [v]\in W[m]_\ell H^m(L_i) .
\end{equation}
Then $[u]$ comes from an element $u\in W_{k-n} L_i^n$, and $[v]$ comes
from $v\in W_{\ell-m} L_i^m$. So
\begin{equation}
\ell_2(u\wedge v) \in W_{(k-n)+(\ell-m)} L_i^{n+m} 
\end{equation}
and this corresponds to
\begin{equation}
[[u], [v]] \in W[n+m]_{k+\ell} H^{n+m}(L_i) .
\end{equation}
This proves that the Lie bracket is a morphism of mixed Hodge
structures.
\end{proof}

\begin{definition}
An \emph{augmented mixed Hodge diagram of Lie algebras} is the data of
a mixed Hodge diagram of Lie algebras $L$ and a Lie algebra
$\mathfrak{g}$ carrying a mixed Hodge structure, seen as a mixed Hodge
diagram of Lie algebras concentrated in degree zero,
together with a morphism
\begin{equation}
\epsilon: L \longrightarrow \mathfrak{g}
\end{equation}
of mixed Hodge diagrams of Lie algebras.
\end{definition}

Now we can state and prove the main theorem of this section.

\begin{theorem}
\label{theorem:MHD-L-infinity}
Let ${\epsilon:L\rightarrow \mathfrak{g}}$ be an augmented mixed Hodge diagram
of Lie algebras.  Assume that in each of the components $L_i$ the filtration $W$
has only non-negative weights (i.e.\ $W_k L_i=0$ for $k<0$) and that the mixed
Hodge structure on $\mathfrak{g}$ is pure of weight zero.  Then the desuspended
mapping cone $C$ of $\epsilon$ with its $L_\infty$ algebra structure constructed
by Fiorenza-Manetti is a mixed Hodge diagram of $L_\infty$ algebras.
\end{theorem}

\begin{proof}
It is practical to consider $\mathfrak{g}$ as a mixed Hodge diagram of Lie
algebras concentrated in degree $0$. So we will write terms $\mathfrak{g}_i^n$
that are zero for $n\neq 0$.  The structure of mixed Hodge complex on $C$ is
written in Proposition~\ref{proposition:mapping-cone-mixed-Hode-complex} (with
$C=\Cone(\epsilon)[-1]$), the axioms to be checked are in
Definition~\ref{definition:MHD-L-infinity} using the operations of $L_\infty$
algebra on $C$ described in Theorem~\ref{theorem:main-Fiorenza-Manetti}.

Let $Q$ be the codifferential on the cofree coalgebra on $C[1]$ which
gives the structure of $L_\infty$ algebra to $C$, with its components
$q_r$ ($r\geq 1$), and $\ell_r$ are the corresponding operations on
$C$.  Up to sign and shift of grading, the operations $q_r$ and
$\ell_r$ are given by the same algebraic formulas.

First check the compatibility for $W$ in some component $C_i$.

For $\ell_1$, which is the differential of the desuspended mapping cone:
take $(x,u)\in W_k C_i^n$, so that $x\in W_k L_i^n$ and
$u\in W_{k+1} \mathfrak{g}_i^{n-1}$.  Then we know that
\begin{equation}
\ell_1(x,u) = (d(x),\ \epsilon(x) - d(u))
\end{equation}

(actually $d=0$ on $\mathfrak{g}_i$).  But $d(x)\in W_k L^{n+1}_i$,
$\epsilon(x)\in W_k \mathfrak{g}_i^n \subset W_{k+1} \mathfrak{g}_i^n$
and $d(u)\in W_{k+1} \mathfrak{g}_i^n$. So one sees that
\begin{equation}
\ell_1(x,u)\in W_k C_i^{n+1} .
\end{equation}

For $\ell_2$, which is the naive bracket: take $(x,u)\in W_k C_i^n$, $(y,v)\in W_\ell C_i^m$, so
that $x\in W_{k} L_i^{n}$, $u\in W_{k+1} \mathfrak{g}_i^{n-1}$,
${y\in W_{\ell} L_i^{m}}$, $v\in W_{\ell+1} \mathfrak{g}_i^{m-1}$. We
want to show that
\[ \ell_2 ((x,u)\wedge (y,v)) \in W_{k+\ell}C_i^{n+m} .\]
For the part $\ell_2(x\wedge y)$ this is given (up to sign) by
$[x,y]$, and it is in $W_{k+\ell} L_i^{n+m}$.  For
$\ell_2(u\otimes y)$ this is given up to sign by $[u,\epsilon(y)]$
which is in $W_{(k+1)+\ell} \mathfrak{g}_i^{(n-1)+m}$.  This proves
the compatibility for $\ell_2$.

Now for the higher operations $\ell_r$ with $r\geq 3$ there is only
one compatibility in the
relation~\eqref{equation:description-higher-brackets-long} to check,
and up to sign this is just an iterated bracket. So take $r$ elements
\begin{equation}
(x_j,u_j)\in W_{k_j} C_i^{n_j}, \quad j=1,\dots,r 
\end{equation}
so that $x_j \in W_{k_j} L^{n_j}$ and
$u_j \in W_{k_j +1} \mathfrak{g}^{n_j - 1}$.  In computing
$\ell_r ((x_1,u_1)\wedge\cdots\wedge (x_r, u_r))$, the only nonzero
part is when we multiply only one of the $x_j$ with the others $u_j$;
call it $x_{s}$. Since $\mathfrak{g}_i$ is concentrated in degree $0$,
this is zero if all $u$ are not of degree zero or if $x_s$ is not of
degree $0$. So we can assume $n_s=0$ and $n_j=1$ for $j\neq s$.  Then
for the iterated bracket, and for a permutation
$\{t_1,\dots,t_{r-1}\}$ of $\{1,\dots,\hat{s},\dots,r\}$,
\begin{multline} 
\label{equation:proof-suspended-cone-mixed-Hodge-diagram-L-infinity}
[u_{t_1},[u_{t_2},\dots,[u_{t_ {r-1}},\epsilon(x_{s})]\dots]] \\ \in
W_{(k_{t_1}+1) + \cdots + (k_{t_{r-1}}+1)+ k_s}
\mathfrak{g}^{(n_{t_1}-1) + \cdots + (n_{t_{r-1}}-1) + n_{s}}_i  \\
= W_{k_1+\cdots + k_r + (r - 1)} \mathfrak{g}_i^{n_1 + \cdots + n_r -
r + 1} = W_{k_1+\cdots + k_r + r - 1}(\mathfrak{g}_i) .
\end{multline}
One would like
\begin{equation}
 \ell_r ((x_1,u_1)\wedge\cdots\wedge (x_r, u_r)) \in W_{k_1 + \cdots
+ k_r} C_i^{n_1 + \cdots + n_r + 2 - r} = W_{k_1 + \cdots + k_r}
C_i^{1}
\end{equation}
so that the iterated bracket
in~\eqref{equation:proof-suspended-cone-mixed-Hodge-diagram-L-infinity}
would land in $W_{k_1+\cdots+k_r+1}(\mathfrak{g}_i)$.  But if we
assume that $\mathfrak{g_i}$ has pure weight zero and since $r\geq 3$,
the condition
\[ W_{k_1+\cdots + k_r + r - 1}(\mathfrak{g_i}) \subset
W_{k_1+\cdots+k_r+1}(\mathfrak{g_i}) \subset \mathfrak{g_i} \]
is realized by an equality as soon as $k_1+\cdots+k_r+1\geq 0$. So,
under the assumption that $L_i$ has only non-negative weights, one can
reduce the compatibility checking to $k_1,\dots,k_r \geq 0$ and this
equality is realized.

The condition to check for $F$ on $C_\CC$ is much easier because there
is no shift in the filtration. One sees directly that
$(x,u)\in F^p C_\CC^n$ means $x\in F^p L_\CC^n$,
$u\in F^p \mathfrak{g}_\CC^{n-1}$, so that $d(x)\in F^p L_\CC^{n+1}$,
$\epsilon(x)\in F^p \mathfrak{g}_\CC^n$ and
$d(u)\in F^p \mathfrak{g}_\CC^n$ so
\begin{equation}
 \ell_1(x,u) \in F^p C_\CC^{n+1} .
 \end{equation}
For $\ell_2$ then $(x,u)\in F^p C_\CC^n$, $(y,v)\in F^q C_\CC^m$ means
that $x\in F^p L_\CC^{n}$, $u\in F^p \mathfrak{g}_\CC^{n-1}$,
${y\in F^q L_\CC^{m}}$, $v\in F^q \mathfrak{g}_\CC^{m-1}$. So
\begin{multline}
\ell_2((x,u)\wedge (y,v)) = \left( [x,y],\
  \frac{1}{2}\ [u,\epsilon(y)]+\frac{(-1)^{|x|}}{2}\,[\epsilon(x), v] \right) \\
\in F^{p+q} L_\CC^{n+m} \oplus F^{p+q} \mathfrak{g}_\CC^{n+m-1} =
F^{p+q} C_\CC^{n+m} .
\end{multline}
Finally for the higher operations, take again
$(x_j,u_j)\in F^{p_j} C_\CC^{n_j}$ so that
$x_j \in F^{p_j} L_\CC^{n_j}$ and
$u_j \in F^{p_j} \mathfrak{g}_\CC^{n_j - 1}$. Again, select one
element $x_s$ and consider a permutation $\{t_1,\dots,t_{r-1}\}$ of
$\{1,\dots,\hat{s},\dots,r\}$. Then
\begin{multline} [u_{t_1},[u_{t_2},\dots,[u_{t_
{r-1}},\epsilon(x_{s})]\dots]] \\ \in F^{p_{t_1} + \cdots + p_{t_{r-1}} +
p_s}  \mathfrak{g}_\CC^{(n_{t_1}-1) + \cdots + (n_{t_{r-1}}-1) + n_{s}} 
= F^{p_1+\cdots + p_r} \mathfrak{g}_\CC^{n_1 + \cdots + n_r - r + 1} .
\end{multline}
This checks directly that
\begin{equation}
\ell_r((x_1,u_1)\wedge\cdots\wedge (x_r,u_r)) \in F^{p_1 + \cdots +
p_r} C_\CC^{n_1+\cdots+n_r+2-r}.
\end{equation}
\end{proof}

Our theorem needs to be completed by a study of quasi-isomorphisms because we
will have to work with mixed Hodge diagrams defined up to quasi-isomorphism.

\begin{lemma}
\label{lemma:quasi-isomorphism-deformation-functor-augmented-infinity-MHD}
Let ${\phi:L\stackrel{\approx}{\longrightarrow} L'}$ be a quasi-isomorphism of mixed Hodge diagrams of
augmented Lie algebras over the same Lie algebra $\mathfrak{g}$ with a mixed
Hodge structure, satisfying both the hypothesis of
Theorem~\ref{theorem:MHD-L-infinity}. Then the induced morphism
${\psi:C\rightarrow C'}$ between the desuspended mapping cones is a
quasi-isomorphism of mixed Hodge diagrams of $L_\infty$ algebras.
\end{lemma}

\begin{remark}
Without taking into account mixed Hodge complexes, the fact that $\psi$ is a
quasi-isomorphism of $L_\infty$ algebras follows from the functoriality of the
structure of Fiorenza-Manetti (showing that $\psi$ is a morphism of $L_\infty$
algebras) and from the five lemma applied to the long exact of the mapping cone
(showing that $\psi$ is a quasi-isomorphism). And by the way one recovers the
Lemma~\ref{lemma:quasi-isomorphism-deformation-functor-augmented} by invariance
of the deformation functor under quasi-isomorphism.
\end{remark}

\begin{proof}[{Proof of Lemma~\ref{lemma:quasi-isomorphism-deformation-functor-augmented-infinity-MHD}}]
By functoriality $\psi$ is already a morphism of mixed Hodge diagrams of
$L_\infty$ algebras. Then the argument is the same as in the above remark but
applying first the functor $\Gr$: in the component $i$ by definition $\phi$
induces a quasi-isomorphism
$\Gr^W_k(L_i)\stackrel{\approx}{\longrightarrow} \Gr^W_k (L'_i)$. So apply the five
lemma to the long exact sequence of the mapping cones of
$\Gr^W_k( L_i)\rightarrow \Gr^W_k( \mathfrak{g})$ and
$\Gr^W_k (L'_i)\rightarrow \Gr^W_k (\mathfrak{g})$ then $\psi$ is a filtered
quasi-isomorphism in the component $i$. Similarly over $\CC$ with the two
filtrations $W,F$ then $\psi$ is a bifiltered quasi-isomorphism.
\end{proof}

\section{Bar construction on mixed Hodge diagrams of \texorpdfstring{$L_\infty$}{L-infinity} algebras}
\label{section:bar-MHD}
Then we show the compatibility of the bar construction with the mixed Hodge
diagrams of $L_\infty$ algebras. We will show that if $L$ is such a mixed Hodge
diagram then $\mathscr{C}(L)$ (defined by applying the functor $\mathscr{C}$ to
each component of the diagram) is a (inductive limit of) mixed Hodge diagram of
coalgebras, and its $H^0$ is then a coalgebra with a mixed Hodge structure.

Recall that $\mathscr{C}(L)$ has a canonical increasing filtration given by
\begin{equation}
\mathscr{C}_s(L) := \bigoplus_{r=1}^s (L[1])^{\odot r}, \quad s\geq 1 
\end{equation}
by sub-DG coalgebras that we call the \emph{bar filtration} and $\mathscr{C}(L)$
is the inductive limit of the $\mathscr{C}_s(L)$. We will work with a fixed
index $s$ and we will always consider $\mathscr{C}(L)$ as an inductive limit.

\begin{definition}
Let $(L,W)$ be a filtered $L_\infty$ algebra (i.e.\ $W$ satisfies
equation~\eqref{equation:MHD-L-infinity-W} of
Definition~\ref{definition:MHD-L-infinity}). The filtration induced by $W$ on
$\mathscr{C}(L)$ via $W[1]$ on $L[1]$ and then by multiplicative extension to
$(L[1])^{\odot r}$ is called the \emph{bar-weight filtration}. We denote it by
$\mathscr{C}W$.
\end{definition}

As in Hain's work \cite[\S~3.2]{Hain}, the bar-weight filtration is a
convolution of the weight filtration and the bar filtration.

\begin{remark}
The bar-weight filtration on $\mathscr{C}(L)$ may not be biregular, however it
is biregular on each $\mathscr{C}_s(L)$ if $L$ is bounded-below because it
involves only a finite number of symmetric products.
\end{remark}

\begin{lemma}
Let $(L,W)$ be a filtered $L_\infty$ algebra, bounded-below. Then
$\mathscr{C}_s(L)$ is a filtered DG coalgebra for the bar-weight filtration.
\end{lemma}

\begin{proof}
As $\mathscr{C}W$ is induced by $W[1]$ on $L[1]$ and then by multiplicative
extension, and seeing the algebraic formula for the coproduct of the cofree
coalgebra (Definition~\ref{definition:cofree-coalgebra}
equation~\eqref{equation:deconcatenation-coproduct}), it is clear that
$\mathscr{C}W$ is compatible with the graded coalgebra structure.  Then we have
to show that the codifferential $Q$ of $\mathscr{C}_s(L)$ respects the
filtration, and it is enough to check it for its components
$q_r:(L[1])^{\odot r}\rightarrow L[1]$ $(r\geq 1$) because of the explicit
formula for recovering $Q$ from its components (\cite[VIII.34]{ManettiLectures}).

So take $r$ elements
\begin{equation}
x_j[1] \in W[1]_{k_j} L_i[1]^{n_j},\quad j=1,\dots,r 
\end{equation}
which means $x_j \in W_{k_j - 1} L_i^{n_j + 1}$. Then
\begin{multline}
\label{eq:q-W-bar}
q_r (x_1[1] \odot \cdots \odot x_r[1]) = \pm
\ell_r(x_1\wedge\cdots\wedge x_r) \\ \in W_{(k_1-1)+\cdots +
(k_r-1)} L_i^{(n_1+1)+\cdots+ (n_r+1)+(2-r)} \\
\subset W_{k_1+\cdots+k_r-1} L_i^{n_1+\cdots+n_r+2}=
W[1]_{k_1+\cdots+k_r} L_i[1]^{n_1+\cdots+n_r+1} .
\end{multline}
This is the desired compatibility.
\end{proof}

If $L$ is a bifiltered $L_\infty$ algebra (i.e.\ equipped with filtration $W,F$
as is $L_\CC$ in Definition~\ref{definition:MHD-L-infinity}) then there is an
induced filtration $F$ on $\mathscr{C}_s(L)$ defined by the induced $F$ (without
shifting) to $L[1]$ and then by multiplicative extension. It is then easier to
see that $\mathscr{C}_s(L)$ is also a filtered coalgebra for $F$:
take $r$ elements
\begin{equation}
x_j[1]\in F^{p_j} L_\CC[1]^{n_j},\quad j=1,\dots,r 
\end{equation}
which means $x_j \in F^{p_j} L_\CC^{n_j +1}$ and then directly
\begin{multline}
q_r(x_1[1]\odot\cdots\odot x_r[1]) = \pm \ell_r(x_1\wedge\cdots\wedge
x_r) \\
\in F^{p_1+\cdots+p_r} L_\CC^{(n_1+1)+\cdots+(n_r+1)+(2-r)} \\
= F^{p_1+\cdots+p_r} L_\CC^{n_1+\cdots+n_r+2} = F^{p_1+\cdots+p_r}
L_\CC[1]^{n_1+\cdots+n_r+1} .
\end{multline}

To sum up, if $L$ is a mixed Hodge diagram of $L_\infty$ algebras whose
components $L_i$ are bounded-below, then $\mathscr{C}_s(L)$ is a diagram
consisting of a filtered DG coalgebra $(\mathscr{C}_s(L_\kk),\mathscr{C}W)$ over
$\kk$ and a bifiltered DG coalgebra $(\mathscr{C}_s(L_\CC),\mathscr{C}W,F)$ over
$\CC$, related by a chain of morphisms
\begin{equation}
(\mathscr{C}_s(L_\kk),\mathscr{C}W) \otimes_\kk \CC
= (\mathscr{C}_s(L_\kk \otimes_\kk \CC),\mathscr{C}W)
\longleftrightarrow (\mathscr{C}_s(L_\CC),\mathscr{C}W)
\end{equation}
of DG coalgebras filtered by $\mathscr{C}W$.

We arrive finally at our main goal. For this we follow closely the method of
Hain \cite[\S~3]{Hain} for commutative DG algebras, re-writing it for $L_\infty$
algebras. This is also re-written in \cite[\S~8.7]{PetersSteenbrink}.

\begin{theorem}
\label{theorem:mixed-Hodge-diagram-bar}
Let $L$ be a mixed Hodge diagram of $L_\infty$ algebras.  Assume that each
component $L_i$ is non-negatively graded ($L_i^n=0$ for $n<0$) and that
$H^0(L)=0$.  Then $\mathscr{C}_s(L)$ is a mixed Hodge diagram of coalgebras for
any $s\geq 1$ and $\mathscr{C}(L)$ is an inductive limit of mixed Hodge diagrams
of coalgebras.
\end{theorem}

\begin{proof}
We need to check the axioms of Definition~\ref{definition:mixed-Hodge-complex}
for $\mathscr{C}_s(L)$. The fact that $\mathscr{C}(L)$ will be an inductive
limit in the category of mixed Hodge diagrams will then be clear.

We fix temporarily a component $L_i$. Since we took the precaution to work with
inductive limits, $\mathscr{C}_s(L_i)$ is a bounded-below complex because it is
obtained by a finite numbers of symmetric powers from $L_i[1]$ which is
bounded-below. Also, the filtration induced by $W$ is biregular.

The axiom~\ref{MHC-axiom-1} is almost checked during the proof of
Theorem~\ref{theorem:pro-representability-theorem}: using transfer of structure
to the cohomology, up to $L_\infty$-quasi-isomorphism (which does not change the
cohomology of $\mathscr{C}_s(L)$), one can assume that $L_i$
has $L_i^n=0$ for $n\leq 0$ and other terms $L_i^n$ finite-dimensional. Then
$L_i[1]^n=0$ for $n<0$ and $\mathscr{C}_s(L_i)$ is a finite sum of a finite
number of symmetric powers of such $L_i[1]$ so is finite-dimensional in each
degree and so is $H^\bullet(\mathscr{C}_s(L_i))$.

To go further we need to compute the spectral sequence for the bar-weight
filtration on $\mathscr{C}_s(L)$ and relate it to the spectral sequence for the
weight filtration on~$L$.

Since $W$ and
$\mathscr{C}W$ are decreasing filtrations, we work with $-k$ instead
of $k$. We denote by $\mathscr{C}_s(L_i)^m$ the component of total
degree $m$ in $\mathscr{C}_s(L_i)$.  Then by definition of the
spectral sequence
\begin{equation}
{}_{\mathscr{C}W} E_0^{-k,q}(\mathscr{C}_s(L_i)) =
\Gr^{\mathscr{C}W}_{k}\mathscr{C}_s(L_i)^{-k+q} . 
\end{equation}
Since
\begin{multline}
\mathscr{C}W_k (L_i[1])^{\odot r} = \bigoplus_{k_1+\cdots+k_s=k}
W[1]_{k_1} L_i[1] \odot \cdots \odot
W[1]_{k_s} L_i[1] \\
= \bigoplus_{k_1+\cdots+k_r=k-r} W_{k_1} L_i[1] \odot \cdots \odot
W_{k_s} L_i[1]
\end{multline}
it follows that 
\begin{equation}
\Gr^{\mathscr{C}W}_{k}\mathscr{C}_s(L_i)^{-k+q} = \bigoplus_{r=1}^s
\Gr^W_{k-r} ((L_i[1])^{\odot r})^{-k+q} = \bigoplus_{r=1}^s
\Gr^W_{k-r} (L_i^{\wedge r})^{-k+q+r}.
\end{equation}
So we recognize
\begin{equation}
\label{eq:bar-mixed-Hodge-E0}
{}_{\mathscr{C}W} E_0^{-k,q}(\mathscr{C}_s(L_i)) = \bigoplus_{r=1}^s {}_W E_0^{-k+r,q}(L_i^{\wedge r}) .
\end{equation}
The differential $d_0$ on
${}_{\mathscr{C}W} E_0^{-k,\bullet}(\mathscr{C}_r(L_i))$ is induced by the
codifferential $Q:=\sum q_r$ of $\mathscr{C}(L_i)$. Crucial here is the
equation~\eqref{eq:q-W-bar} appearing in the preceding proof, which shows that
$q_r$ is zero on $\Gr^{\mathscr{C}W}_{k}\mathscr{C}(L_i)$ for all $r\geq
2$.
Thus $d_0$ is only induced by $q_1$, which is up to sign $d[1]$, and it is the
sum of the $d_0$'s appearing on the right side
of~\eqref{eq:bar-mixed-Hodge-E0}. But by
Proposition~\ref{proposition:tensor-product-mixed-Hodge-complex} $L^{\wedge r}$ is
a mixed Hodge complex and this right side is a direct sum of terms ${}_W E_0$ of
mixed Hodge complexes.

This computation allows us to check that the comparisons morphisms are
quasi-isomorphisms. Let
\begin{equation}
\phi: (L_i,W) \stackrel{\approx}{\longrightarrow} (L_j,W)
\end{equation}
be some comparison morphism between the two components $L_i, L_j$,
which by hypothesis is a filtered quasi-isomorphism. By the Künneth
formula (combined with the fact that we work with bounded below
complexes) $\phi$ induces a filtered quasi-isomorphism
\begin{equation}
 ((L_i)^{\wedge r}, W^{\wedge r}) \stackrel{\approx}{\longrightarrow}
((L_j)^{\wedge r}, W^{\wedge r})
\end{equation}
 so it induces an isomorphism
 \begin{equation}
 {}_W E_0^{-k+r,q}(L_i^{\wedge r}) \stackrel{\simeq}{\longrightarrow}
{}_W E_0^{-k+r,q}(L_j^{\wedge r}) .
\end{equation}
So equation~\eqref{eq:bar-mixed-Hodge-E0} tells us precisely that
$\phi$ induces a filtered quasi-isomorphism
\begin{equation}
 (\mathscr{C}_s(L_i), \mathscr{C}W)
\stackrel{\approx}{\longrightarrow} (\mathscr{C}_s(L_j),
\mathscr{C}W).
\end{equation}

Then we can check the axiom~\ref{MHC-axiom-2}, in the component over $\CC$ carrying
also the filtration $F$.  By this axiom for $L_\CC^{\wedge r}$ the differential
of ${}_W E_0^{-k+r,\bullet}(L_\CC^{\wedge r})$ is strictly compatible with the
induced filtration $F$, so from equation~\eqref{eq:bar-mixed-Hodge-E0} the
differential of ${}_{\mathscr{C}W} E_0^{k,\bullet}(\mathscr{C}_r(L_\CC))$ is the
direct sum of these and is also strictly compatible with $F$.

Finally to check the axiom~\ref{MHC-axiom-3} we compute the spectral sequence at the
page $E_1$. By definition
\begin{equation}
 {}_{\mathscr{C}W} E_1^{-k,q}(\mathscr{C}_s(L_i)) =
H^{-k+q}(\Gr^{\mathscr{C}W}_{k} \mathscr{C}_s(L_i)) 
\end{equation}
where the cohomology is computed with respect to $d_0$.  Then
using~\eqref{eq:bar-mixed-Hodge-E0}
\begin{equation}
 H^{-k+q}(\Gr^{\mathscr{C}W}_{k} \mathscr{C}_s(L_i)) =
\bigoplus_{r=1}^s H^{-k+q}(\Gr^W_{k-r}(L_i^{\wedge r})^{\bullet+r}) =
\bigoplus_{r=1}^s {}_W E_1^{-k+r,q}(L_i^{\wedge r}) .
\end{equation}
So, put together,
\begin{equation}
\label{eq:bar-mixed-Hodge-E1}
{}_{\mathscr{C}W} E_1^{-k,q}(\mathscr{C}_s(L_i)) =
\bigoplus_{r=1}^s {}_W E_1^{-k+r,q}(L_i^{\wedge r}) .
\end{equation}
Since $L^{\wedge r}$ is a mixed Hodge complex, in
equation~\eqref{eq:bar-mixed-Hodge-E1} the terms on the right side
\begin{equation}
{}_W E_1^{-k+r,q}(L_i^{\wedge r}) = H^{-k+r+q}(\Gr^W_{k-r} L_i^{\wedge r}) 
 \end{equation}
define, when varying $i$, a pure Hodge structure of weight $q$. So does their
direct sum and this proves that the terms
\begin{equation}
 {}_{\mathscr{C}W} E_1^{-k,q}(\mathscr{C}_s(L_i)) =
H^{-k+q}(\Gr^{\mathscr{C}W}_{k} \mathscr{C}_s(L_i)) 
\end{equation}
define a pure Hodge structure of weight $q$.
\end{proof}

As for the main theorem of the preceding section we need to complete our
theorem by a study of quasi-isomorphisms. It is already clear that this
construction is functorial.

\begin{lemma}[{Compare with
Thm.~\ref{theorem:L-infinity-quasi-isomorphism}}]
\label{lemma:quasi-isomorphism-MHD-L-infinity}
Let ${\phi:L\stackrel{\approx}{\longrightarrow} L'}$ be a quasi-isomorphism of
mixed Hodge diagrams of $L_\infty$ algebras satisfying both the hypothesis of
Theorem~\ref{theorem:mixed-Hodge-diagram-bar}. Then the induced morphism
$\mathscr{C}(\phi)$ is a quasi-isomorphism of mixed Hodge diagrams of
coalgebras.
\end{lemma}

\begin{proof}
By the functoriality and explicit nature of the bar construction, it is already
clear that $\mathscr{C}(\phi)$ is a morphism of diagrams of filtered DG
coalgebras, compatible with the bar filtration. Then, following the proof of the
preceding theorem, we see that in the component $i$ and in
equation~\eqref{eq:bar-mixed-Hodge-E1} the hypothesis tells us that $\phi_i$
induces an isomorphism on the right-hand side, so it induces an isomorphism on
the left-hand side. Similarly for the bifiltered part we repeat the arguments
replacing $L_\CC$ by $\Gr^p_F(L_\CC)$.
\end{proof}

Let us sum up what we will need.

\begin{corollary}
\label{corollary:mixed-Hodge-diagram-bar-conclusion}
If $L$ is a mixed Hodge diagram of $L_\infty$ algebras satisfying the hypothesis
of Theorem~\ref{theorem:mixed-Hodge-diagram-bar} then the pro-Artin algebra
\begin{equation}
R := \kk \oplus H^0(\mathscr{C}(L))^*
\end{equation}
of the pro-representability Theorem~\ref{theorem:pro-representability-theorem}
has a mixed Hodge structure, functorial in $L$, independent of $L$ up to
quasi-isomorphism.
\end{corollary}

\begin{proof}
Since $\mathscr{C}_s(L)$ is a mixed Hodge diagram, its $H^0$ has a mixed Hodge
structure and so has its dual. This defines a mixed Hodge structure on the
coalgebra $H^0(\mathscr{C}(L))$ and a mixed Hodge structure on the pro-Artin
algebra $R$.
\end{proof}

\section{Construction of mixed Hodge diagrams: the compact case}
\label{section:geo-compact}

In the two following sections we present several different geometric situations
concerning a complex manifold $X$ and a representation $\rho$ of its fundamental
group into a linear algebraic group $G$. In each of them we construct an
appropriate augmented mixed Hodge diagram of Lie algebras that controls the
deformation theory of $\rho$. Then the
machinery we developed in the two preceding sections gives us directly and
functorially a mixed Hodge structure on the complete local ring $\Ohat_\rho$ of
the representation variety $\Hom(\pi_1(X,x),G)$ at $\rho$.

The compact case is much easier to deal with because the construction of a mixed
Hodge diagram over $\RR$ computing the cohomology of a variety is
straightforward using the algebra of differential forms. So let $X$ be a compact
Kähler manifold, for example a smooth complex projective algebraic variety.

\begin{definition}
\label{definition:variation-Hodge-structure}
A real \emph{polarized variation of Hodge structure} of weight $k$ on $X$ is the
data of a local system of finite-dimensional real vector spaces $V$ on $X$ with
a decreasing filtration of the associated holomorphic vector bundle by
holomorphic sub-vector bundles $\mathcal{F}^\bullet \subset
V\otimes\mathcal{O}_X$, a flat bilinear map $Q : V\otimes V \longrightarrow
\RR$, and a flat connection $\nabla: V\otimes\mathcal{O}_X \longrightarrow V\otimes\Omega^1_X$
such that at each point $x\in X$ the data $(V_x, \mathcal{F}^\bullet_x, Q_x)$
forms a polarized Hodge structure of weight $k$. Furthermore $\nabla$ is
required to satisfy Griffiths' transversality $\nabla(\mathcal{F}^p) \subset \mathcal{F}^{p-1} \otimes \Omega^1_X$.
\end{definition}

Let $x$ be a base point of $X$.  Let
\begin{equation}
\rho:\pi_1(X,x) \longrightarrow GL(V_x)
\end{equation}
be a representation which is the monodromy of a real polarized variation of
Hodge structure $(V,\mathcal{F}^\bullet,\nabla,Q)$ of weight $k$ on $X$.  The
local system of Lie algebras associated to $\rho$ is now $\End(V)$, which by the
usual linear algebraic constructions is a polarized variation of Hodge
structures of weight zero. Explicitly
\begin{equation}
\mathcal{F}^p \End(V\otimes\mathcal{O}_X) = \big\{ f:V\otimes\mathcal{O}_X
\rightarrow V\otimes\mathcal{O}_X \ \big|\ f(\mathcal{F}^\bullet) \subset
\mathcal{F}^{\bullet + p} \big\}.
\end{equation}
One constructs a real mixed Hodge diagram as follows.  Let
\begin{equation}
L_\RR := \mathscr{E}^\bullet(X, \End(V))
\end{equation}
and
\begin{equation}
L_\CC := \mathscr{E}^\bullet(X, \End(V\otimes\CC)) .
\end{equation}
One defines a filtration $W$ which is the trivial one (everything has weight
zero) and a filtration $F$ on $L_\CC$ by
\begin{equation}
 F^p L_\CC^n := \bigoplus_{r+s \geq p} \mathscr{E}^{r,n-r}(X,
\mathcal{F}^s \End(V)).
\end{equation}
The DG Lie algebra structure is given as usual
in~\eqref{equation:differential-L} and~\eqref{equation;Lie-bracket-L}.

Then we define an augmentation $\epsilon_x$ from $(L_\RR,L_\CC,W,F)$ to the
mixed Hodge diagram of Lie algebras formed by the Hodge structure on the Lie
algebra $\mathfrak{g}:=\End(V_x)$ with
\begin{equation}
\mathfrak{g}_\CC=\End(V_x \otimes\CC)=\End(V_x)\otimes\CC
\end{equation}
and the Hodge filtration is simply $F^\bullet=\mathcal{F}^\bullet_x$, by
evaluating forms of degree zero at $x$ as in~\eqref{equation:augmentation-L}.

\begin{theorem}
\label{theorem:compact-L}
The data
\begin{equation}
(L_\RR, L_\CC, W, F)
\end{equation}
forms a real mixed Hodge diagram of Lie algebras. Together with
${\epsilon_x:L\rightarrow\mathfrak{g}}$ this is an augmented mixed Hodge diagram
of Lie algebras satisfying the hypothesis of
Theorem~\ref{theorem:MHD-L-infinity}.
\end{theorem}

\begin{proof}
The fact that $L$ is a real mixed Hodge complex is essentially the classical
Hodge theory with values in a variation of Hodge structure of Deligne-Zucker
\cite[\S~2]{Zucker} and follows from the Kähler identities with twisted
coefficients. And by construction the Lie bracket and $\epsilon_x$ are compatible
with the filtrations.
\end{proof}

So we apply for the first time the method we developed.

\begin{theorem}
\label{theorem:main-compact}
If $X$ is a compact Kähler manifold and $\rho$ is the monodromy of a real
polarized variation of Hodge structure $V$ on $X$, then there is a real mixed
Hodge structure on the complete local ring $\Ohat_\rho$ of the representation
var\-ie\-ty $\Hom(\pi_1(X,x), GL(V_x))$ at $\rho$ which is functorial in $X,x,\rho$.
\end{theorem}

\begin{proof}
Over both fields $\RR$ and $\CC$, $L$ and its augmentation control the
deformation theory of $\rho$: this is the main theorem of Goldman-Millson
(Theorem~\ref{theorem:Goldman-Millson-main-augmented}). We will say that the
functor of deformations of $\rho$ is controlled by the mixed Hodge diagram of
augmented Lie algebras $L$. By
Lemma~\ref{lemma:main-augmented-deformation-Manetti-Eyssidieux-Simpson}, this
deformation functor is associated with the $L_\infty$ algebra structure on the
desuspended mapping cone $C$ of $\epsilon_x$.

We need to check that $H^0(C)=0$. By definition of the cone, a closed element of
$C^0$ is given by a $\mathcal{C}^\infty$ section $\omega$ of $\End(V)$ such that
$d(\omega)=0$, so that $\omega$ is locally constant, and such that
$\omega(x)=0$. So $\omega=0$ globally.

Then the deformation functor of $\rho$ is pro-represented by a pro-Artin algebra
that we denote by $R$ as in
Theorem~\ref{theorem:pro-representability-theorem}. By the pro-Yoneda lemma
this $R$ is canonically isomorphic to
$\Ohat_\rho$.  Again all this construction commutes with the change of base
field so we work with $C$ as mixed Hodge diagram. The augmented mixed Hodge
diagram of Lie algebras we just have defined satisfies the hypothesis of our
Theorem~\ref{theorem:MHD-L-infinity}. So $C$ is a mixed Hodge diagram of
$L_\infty$ algebras.  Then we apply the Theorem~\ref{theorem:mixed-Hodge-diagram-bar} (or its
Corollary~\ref{corollary:mixed-Hodge-diagram-bar-conclusion}) to get a mixed
Hodge structure on the pro-Artin algebra $R$, which as we said is canonically
isomorphic to $\Ohat_\rho$ (as pro-Artin algebra, so that each quotients by
powers of the maximal ideals are isomorphic). And this induces the mixed Hodge
structure on $\Ohat_\rho$.
\end{proof}

With the same method of proof, there are several possible variations. First one
can work with representations taking values in a real linear algebraic group.

\begin{proposition}
Let $G$ be a real linear algebraic group.  Let
${\rho:\pi_1(X,x)\rightarrow G(\RR)}$ be a representation which is the monodromy
of a real polarized variation of Hodge structure on $X$. Then there is a
functorial mixed Hodge structure on the local ring $\Ohat_\rho$ of the
representation variety $\Hom(\pi_1(X,x), G)$ at $\rho$.
\end{proposition}

\begin{proof}
Re-write the proof of Theorem~\ref{theorem:main-compact} by replacing $GL(V_x)$
by $G$ and $\End(V_x)$ by the Lie algebra $\mathfrak{g}$ of $G$. Then again one
gets $L$ that is an augmented mixed Hodge diagram  Lie algebras over
$\mathfrak{g}$.
\end{proof}

One can also work with mixed Hodge structures over $\CC$, as in
\cite{EyssidieuxSimpson}.

\begin{definition}[{\cite[1.1]{EyssidieuxSimpson}}]
A \emph{complex Hodge structure} of weight $k$ is the data of a finite-dimensional vector space $K$
over $\CC$ with two filtrations $F, \overline{G}$, such that $K$ decomposes as
$K=\bigoplus_{p+q=k} K^{p,q}$ with
\begin{equation}
F^p K = \bigoplus_{p'\geq p} K^{p',q}, \quad \overline{G}^q K =
\bigoplus_{q'\geq q} K^{p,q'} .
\end{equation}

A \emph{complex mixed Hodge structure} is the data of a finite-dimensional vector
space $K$ over $\CC$ with two decreasing filtrations $F,\overline{G}$ and an
increasing filtration $W$ such that each graded part $\Gr^W_k (K)$ with the
induced filtrations $F,\overline{G}$ is a complex Hodge structure of weight
$k$.

We refer to \cite{EyssidieuxSimpson} for the definitions of polarization
(Def.~1.1) and of variation of Hodge structure (Def.~1.8) in this context.
\end{definition}

For example if $K$ is a mixed Hodge structure over $\kk\subset\RR$ then $K\otimes\CC$ is
canonically a mixed Hodge structure over $\CC$ with $\overline{G}$ being the
conjugate filtration of $F$. It is polarized if $K$ is.

So one can state the most general result:

\begin{proposition}
Let $G$ be a complex linear algebraic group. Let
$\rho:\pi_1(X,x)\rightarrow G(\CC)$ be a representation which is the monodromy
of a complex polarized variation of Hodge structure on $X$. Then there is a
functorial complex mixed Hodge structure on the local ring $\Ohat_\rho$ of the
representation variety $\Hom(\pi_1(X,x), G)$ at $\rho$.
\end{proposition}

\begin{proof}
Define a \emph{complex mixed Hodge complex} to be the data of a DG vector space
$K$ over $\CC$ equipped with filtrations $W,F,\overline{G}$, satisfying the
usual axioms of mixed Hodge complex
(Definition~\ref{definition:mixed-Hodge-complex}): in the
axiom~\ref{MHC-axiom-2} we require the differential of $\Gr_k^W(K)$ to be
strictly compatible with \emph{both} filtrations $F,\overline{G}$ and in the
axiom~\ref{MHC-axiom-3} we require each term $H^n(\Gr_k^W(K))$ to carry a pure
complex Hodge structure with the induced filtrations $F,\overline{G}$.

Then each term of the cohomology of a complex mixed Hodge complex carries a
complex mixed Hodge structure.  So we can re-write everything with complex mixed
Hodge structures and complex mixed Hodge complexes.
\end{proof}

We can also show that the mixed Hodge structure on $\Ohat_\rho$ is defined over
$\kk\subset\RR$ if $\rho$ is (i.e.\ if the polarized variation of Hodge
structure whose monodromy is $\rho$ and the algebraic group $G$ are defined over
$\kk$).

\begin{proposition}
\label{proposition:main-compact-over-k}
In Theorem~\ref{theorem:main-compact} assume that the variation of Hodge
structure $V$ is defined over $\kk$. Then the mixed Hodge structure on
$\Ohat_\rho$ is defined over $\kk$.
\end{proposition}

However we will prove this only after studying the non-compact case. For this
one needs to construct mixed Hodge diagrams over $\QQ$ and this is rational
homotopy theory.

\section{Construction of mixed Hodge diagrams: the non-compact case}
\label{section:geo-non-compact}

In the case where $X$ is non-compact, the construction of an appropriate mixed
Hodge diagram that computes the cohomology of $X$ is more difficult and depends
on the choice of a compactification of $X$.

So let $X$ be a smooth quasi-projective algebraic variety over $\CC$. Let $x$ be
a base point of $X$. Let $G$ be a linear algebraic group over $\kk\subset\RR$
with Lie algebra $\mathfrak{g}$. Let
\begin{equation}
\rho : \pi_1(X,x) \longrightarrow G(\kk)
\end{equation}
be a representation and we assume that $\rho$ has \emph{finite} image. Under
these hypothesis the ideas to construct a controlling mixed Hodge diagram of Lie
algebras are already entirely present in the work of Kapovich-Millson
\cite[\S~14--15]{KapovichMillson}. However they rely strongly on the theory of
minimal models of Morgan \cite{Morgan} which is not completely functorial.
So we re-write these ideas using the more powerful
construction of mixed Hodge diagrams of Navarro Aznar \cite{Navarro}.

Let us first explain briefly the ideas and the notations.
Let the finite group
\begin{equation}
\Phi := \frac{\pi_1(X,x)}{\Ker(\rho)} \simeq \rho(\pi_1(X,x)).
\end{equation}
To $\Ker(\rho)\subset\pi_1(X,x)$ corresponds a finite étale Galois cover
$\pi:Y\rightarrow X$ with automorphism group $\Phi$ that acts simply transitively on
the fibers, and equipped with a fixed base point $y\in Y$ over $x$. It is known
that $Y$ is automatically a smooth quasi-projective algebraic variety. The flat
principal bundle $P$ induced by the holonomy of $\rho$ is trivial when
pulled-back to $Y$, as well as its adjoint bundle $\Ad(P)$. So the DG Lie
algebra of Goldman-Millson is (over $\RR$ or $\CC$)
\begin{equation}
\label{equation:equivariant-GM-finite-image}
L := \mathscr{E}^{\bullet}(X, \Ad(P)) = (\mathscr{E}^{\bullet}(Y,\pi^*\Ad(P)))^{\Phi} =
(\mathscr{E}^{\bullet}(Y)\otimes\mathfrak{g})^{\Phi}
\end{equation}
(where the exponent $\Phi$ denotes the invariants by the action of $\Phi$). In
order to construct a mixed Hodge diagram that is quasi-isomorphic to this we
simply want to find a mixed Hodge diagram for $Y$ equipped with an action of
$\Phi$, then tensor it with $\mathfrak{g}$, then take the invariants by $\Phi$.

For the augmentation, since there is a canonical identification of fibers
$\Ad(P)_x \simeq \mathfrak{g}$ one can define an augmentation
\begin{equation}
\label{equation:non-compact-augmentation-epsilon}
\epsilon_x : \mathscr{E}^\bullet(X, \Ad(P)) \longrightarrow \mathfrak{g}
\end{equation}
exactly as in the compact case~\eqref{equation:augmentation-L} by evaluating
degree zero forms at $x$ and higher degree forms to zero.  This augmentation can
be lifted equivariantly to $Y$: $\epsilon_x$ comes from the augmentation
\begin{equation}
\label{equation:lifted-augmentation}
\eta_x: \mathscr{E}^\bullet(Y)\otimes\mathfrak{g} \longrightarrow \mathfrak{g} 
\end{equation}
defined by 
\begin{equation}
\label{equation:lifted-augmentation-formula}
\eta_x(\omega\otimes u) := \frac{1}{|\Phi|} \sum_{g\in\Phi} \epsilon_y(g.(\omega\otimes u))
\end{equation}
where $\epsilon_y$ simply evaluates forms with values in $\mathfrak{g}$ at $y$. Observe the
notations: $\epsilon_y$ depends on $y$ but in the definition of $\eta_x$ we sum over the whole
(finite) fiber of $\pi$ over $x$ so $\eta_x$ depends only on $x$.
Then we see that $\eta_x$ induces $\epsilon_x$ when restricted to the
equivariant forms $(\mathscr{E}^\bullet(Y)\otimes\mathfrak{g})^\Phi$.

We will also need an \emph{equivariant completion} of $Y$. By the theorem of
Sumihiro \cite{Sumihiro} it is possible to compactify
${Y\hookrightarrow \overline{Y}'}$ so that the action of $\Phi$ extends to
$\overline{Y}'$. And then by the work of Bierstone-Milman on canonical
resolutions of singularities \cite[\S~13]{BierstoneMilman} one can construct a
resolution of singularities ${\overline{Y}\rightarrow\overline{Y}'}$ to which
the action of $\Phi$ lifts. This $\overline{Y}$ we call an equivariant
completion of $Y$ and $D:=\overline{Y}\setminus Y$ is a divisor with simple
normal crossings on which $\Phi$ acts.

From the data of a smooth quasi-projective variety $Y$ with a smooth
compactification $\overline{Y}$ by a divisor with normal crossings $D$, Navarro
Aznar \cite{Navarro} has constructed a functorial mixed Hodge diagram of
commutative algebras computing the cohomology algebra of $Y$ over
$\kk\subset\RR$ together with its mixed Hodge structure. Let us denote by
$\MHD(\overline{Y},D)_\kk$ this diagram. We will only need to know that each
component is related by a canonical chain of quasi-isomorphisms to the usual DG
algebra computing the cohomology of $Y$: the component over $\CC$ is related to
the usual differential forms on $Y$ via the holomorphic forms with logarithmic
poles along $D$, denoted by $\Omega^\bullet_{\overline{Y}}(\log D)$, and the
component over $\kk$ is related to the usual singular cochain complex over
$\kk$.

Since these mixed Hodge diagrams are functorial, in our situation the group
$\Phi$ acts on $\MHD(\overline{Y},D)_\kk$. We always denote by an exponent
$\Phi$ the diagram formed by the elements invariant under the action of $\Phi$.

\begin{lemma}
The diagram of invariants $\MHD(\overline{Y},D)_\kk^\Phi$ is again a mixed Hodge
diagram that computes canonically the cohomology of $X$ with its mixed Hodge
structure.
\end{lemma}

\begin{proof}
This follows essentially from the fact that taking the invariants by a finite
group commutes with cohomology. Hence the cohomology of $X$ is given by the
invariant cohomology of $Y$. Since $\Phi$ acts by morphisms of mixed Hodge
diagrams (i.e.\ preserves the structures of DG algebras and the filtrations) it
is easy to check that $\MHD(\overline{Y},D)_\kk^\Phi$ is again a mixed Hodge
diagram that computes this invariant cohomology of $Y$.
\end{proof}

Now equation~\eqref{equation:equivariant-GM-finite-image} and the related
remarks explain how to construct the controlling DG Lie algebra of
Goldman-Millson. In each component of $\MHD(\overline{Y},D)_\kk$, which is a
commutative DG algebra, one tensors by $\mathfrak{g}$ which is defined over
$\kk$. Then the group $\Phi$ acts on both the DG algebra and on $\mathfrak{g}$
and we take the invariants. We denote this by
\begin{equation}
\label{equation:non-compact-definition-L}
M := \big( \MHD(\overline{Y},D)_\kk \otimes \mathfrak{g} \big)^\Phi .
\end{equation}

\begin{lemma}
\label{lemma:non-compact-L}
This $M$ is a mixed Hodge diagram of Lie algebras that is canonically
quasi-isomorphic to the DG Lie algebra $L$ of Goldman-Millson controlling the
deformation theory of $\rho$.
\end{lemma}

\begin{proof}
First, each component of $M$ is a DG Lie algebra. The fact that both taking
cohomology and taking the graded pieces of filtrations commute with both the
tensor product with $\mathfrak{g}$ (which is concentrated in degree zero and has
trivial filtrations) and with the invariants by $\Phi$ (which is a finite group
that acts preserving all these structures) implies easily that $M$ is a mixed
Hodge diagram of Lie algebras.

Then follow the equation~\eqref{equation:equivariant-GM-finite-image}: since
$\MHD(\overline{Y},D)_\kk$ computes equivariantly the cohomology of $Y$ then
$\MHD(\overline{Y},D)_\kk \otimes \mathfrak{g}$ computes the cohomology of $Y$
with coefficients in $\mathfrak{g}$ and $M$ computes the cohomology of $X$ with
twisted coefficients in $\Ad(P)$.
\end{proof}

\begin{remark}
\label{remark:non-compact-L-quasi-isomorphism}
We see from the proof that a quasi-isomorphism $\phi$ between two such mixed
Hodge diagrams associated with two equivariant compactifications
$\overline{Y}, \overline{Y}'$, with $\phi$ equivariant with respect to $\Phi$,
induces a quasi-isomorphism $M\stackrel{\approx}{\longrightarrow} M'$ between the corresponding mixed
Hodge diagrams of Lie algebras.
\end{remark}

The last step is to construct the augmentation $\epsilon_x$ at the level of $M$
and not $L$. Essentially this follows from the fact that the evaluation of
differential forms on $Y$ at a base point $y$ can be defined
sheaf-theoretically: it is induced by the natural morphism from the constant
sheaf $\kk_Y$ to the skyscraper sheaf $\kk_y$ which in some sense evaluates
sections at $y$. However at this point we need to enter more into the detailed
construction of Navarro Aznar.

Recall that the mixed Hodge diagrams of Navarro Aznar are obtained in two
steps. First one defines sheaves of commutative DG algebras equipped with
filtrations, forming a \emph{mixed Hodge diagram of sheaves of commutative
algebras}. Then one applies the \emph{Thom-Whitney functor} $R_{\TW}\Gamma$ which
is quasi-isomorphic to the usual derived functor of global sections $R\Gamma$
and gives here a mixed Hodge diagram of commutative algebras.

The component over $\CC$ of $\MHD(\overline{Y},D)_\kk$ is the
$R_{\TW}(\overline{Y},-)$ of a certain sheaf of analytic differential forms on
$\overline{Y}$ with logarithmic poles along $D$ called the \emph{logarithmic
Dolbeaut complex} (\cite[\S~8]{Navarro}) denoted by
$\mathscr{A}^\bullet_{\overline{Y}}(\log D)_\CC$. So it is naturally equipped with
an augmentation
\begin{equation}
\mu_y : \mathscr{A}^\bullet_{\overline{Y}}(\log D)_\CC \longrightarrow \CC_y 
\end{equation}
which by functoriality induces a morphism
\begin{equation}
R_{\TW}(\mu_y) : R_{\TW}\Gamma\big(\overline{Y},\mathscr{A}^\bullet_{\overline{Y}}(\log D)_\CC\big) \longrightarrow R_{\TW}\Gamma(\overline{Y}, \CC_y).
\end{equation}
The left-hand side is by definition the component over $\CC$ of
$\MHD(\overline{Y},D)_\kk$ and the right-hand side is simply $\CC$ itself.
Similarly, the component over $\RR$ comes from a certain sheaf
$\mathscr{A}_{\overline{Y}}^\bullet(\log D)_\RR$ of real analytic
differential forms on $\overline{Y}$ with logarithmic poles along $D$ and comes
equipped with a natural augmentation to $\RR$
\begin{equation}
R_{\TW}(\mu_y): R_{\TW}\Gamma\big(\overline{Y},\mathscr{A}_{\overline{Y}}^\bullet(\log
D)_\RR\big) \longrightarrow R_{\TW}\Gamma(\overline{Y}, \RR_y)\simeq \RR.
\end{equation}
For the component over $\kk$: let us denote by $j:Y\hookrightarrow\overline{Y}$
the compactification
then the sheaf of commutative DG algebras forming the mixed Hodge diagram of
sheaves on $\overline{Y}$ is $R_{\TW}j_* \kk_Y$ (which is quasi-isomorphic to the
usual $Rj_* \kk_Y$). It is canonically equipped with an augmentation $\mu_y$ to
$\kk_y$, which induces
\begin{equation}
R_{\TW}(\mu_y):R_{\TW}\big(\overline{Y}, R_{\TW}j_* \kk_Y \big) \longrightarrow
R_{\TW}(\overline{Y}, \kk_y) \simeq \kk .
\end{equation}
All this defines an augmentation of the mixed Hodge diagram
\begin{equation}
\mu_y : \MHD(\overline{Y}, D)_\kk \longrightarrow \kk .
\end{equation}
Then we mimic~\eqref{equation:lifted-augmentation-formula}:
we define an augmentation
\begin{equation}
\nu_x: \MHD(\overline{Y}, D)_\kk \otimes \mathfrak{g} \longrightarrow \mathfrak{g}
\end{equation}
as
\begin{equation}
\label{equation:non-compact-augmentation-sum}
\nu_x(\omega \otimes u) := \frac{1}{|\Phi|} \sum_{g\in \Phi}
(R_{\TW}(\mu_y)\otimes{\id_{\mathfrak{g}}})(g.(\omega\otimes u))
\end{equation}
where again, in each component, $g$ acts on both the mixed Hodge diagram and on $\mathfrak{g}$.
When restricted to the invariants by $\Phi$, this induces an augmentation (still
denoted by $\nu_x$)
\begin{equation}
\label{equation:non-compact-augmentation}
\nu_x: M=\big(\MHD(\overline{Y},D)_\RR \otimes \mathfrak{g}\big)^\Phi \longrightarrow \mathfrak{g} .
\end{equation}

\begin{lemma}
\label{lemma:non-compact-augmentation}
Via the canonical chain of quasi-isomorphisms relating $M$ to $L$, this $\nu_x$
corresponds to the augmentation $\epsilon_x$
of~\eqref{equation:non-compact-augmentation-epsilon}.
\end{lemma}

By this we mean that the whole canonical chain of quasi-isomorphisms relating
$M$ to $L$ has augmentations to $\mathfrak{g}$, commuting with the
quasi-isomorphisms, and relating $\nu_x$ to $\epsilon_x$.

\begin{proof}
First on $Y$, $\epsilon_y$ is also
induced by the augmentation at the level of sheaves
\[ \mathscr{E}_{Y,\kk}^\bullet \longrightarrow \kk_y \]
then by taking global sections and tensoring with $\mathfrak{g}$. This
augmentation commutes with the chain of quasi-isomorphisms relating
$\mathscr{E}_{Y}^\bullet$ to $\mathscr{A}^\bullet_{\overline{Y}}(\log D)$, via
the intermediate augmentation of $\Omega_{\overline{Y}}^\bullet(\log D)$ which
is defined by the same obvious way, evaluating degree zero holomorphic forms at
$y$ (important is the fact that $y$ is in $Y$ and not on $D$). This is enough to
prove the claim on $Y$, since then it is easy to tensor all the chain of
augmented quasi-isomorphisms by~$\mathfrak{g}$.

One goes from $Y$ to $X$ by simply comparing the
formulas~\eqref{equation:non-compact-augmentation-sum}
and~\eqref{equation:lifted-augmentation-formula}, from which we see by
construction that $\nu_x$ corresponds to the augmentation that we denoted by
$\eta_x$, and then by going to the invariants under $\Phi$ we see that $\nu_x$
on $M$ corresponds to $\epsilon_x$ on $L$.
\end{proof}

Let us sum up.

\begin{theorem}
The data of $M$ and $\epsilon_x$ is an augmented mixed Hodge diagram of Lie
algebras over $\kk$ quasi-isomorphic to the augmented DG Lie algebra of
Goldman-Millson controlling the deformation theory of $\rho$. Up to
quasi-isomorphism, it depends only on $X,x,\rho$.
\end{theorem}

\begin{proof}
Combining Lemma~\ref{lemma:non-compact-L} and
Lemma~\ref{lemma:non-compact-augmentation} forms the first part of the claim. We
need to prove the independence of $M$ on $\overline{Y}$ up to
quasi-isomorphism. By the usual Galois correspondence for covering spaces $Y,y$
depend already only on $X,x,\rho$. And the augmentation does not depend on
$\overline{Y}$. So as in Remark~\ref{remark:non-compact-L-quasi-isomorphism} it
is enough to show that $\MHD(\overline{Y},D)_\kk$ is independent of
$\overline{Y}$ up to quasi-isomorphism. The argument is well-known
(\cite[3.2.II.C]{DeligneII}) except that we work with equivariant
compactifications. So let $\overline{Y}',\overline{Y}''$ be two equivariant
compactifications of $Y$. We look for a third compactification $\overline{Y}$
which dominates both, i.e.\ with two morphisms of pairs
\begin{equation}
\label{equation:dominant-compactification-Y}
 \xymatrix{ & (\overline{Y},D) \ar[ld]_{j'} \ar[rd]^{j''} & \\
(\overline{Y}', D') & & (\overline{Y}'', D'') . }
\end{equation}
This $\overline{Y}$ can be obtained as a resolution of singularities of the
closure of the image of the diagonal embedding of $Y$ into
$\overline{Y}' \times \overline{Y}''$. But by invoking again the combination of
the theorem of Sumihiro on equivariant completion combined with the theorem of
Bierstone-Milman on canonical resolutions of singularities, one can find such an
$\overline{Y}$ to which the action of $\Phi$ lifts. So in the
diagram~\eqref{equation:dominant-compactification-Y} the compactification
$\overline{Y}$ dominates the two others as equivariant compactifications. Then
$j'$ and $j''$ both induce quasi-isomorphisms of mixed Hodge diagrams.
\end{proof}

And our conclusion of this section.

\begin{theorem}
Let $X$ be a smooth complex quasi-projective algebraic variety. Let
$\rho:\pi_1(X,x)\rightarrow G(\kk)$ be a representation with finite image into a
linear algebraic group over the field $\kk\subset\RR$. Then there is a mixed
Hodge structure on $\Ohat_\rho$ defined over $\kk$ that is functorial in
$X,x,\rho$.
\end{theorem}

\begin{proof}
From this data we constructed in the previous theorem an augmented mixed Hodge
diagram of Lie algebras over $\kk$ controlling the deformation theory of
$\rho$. It is independent of $\overline{Y}$ up to quasi-isomorphism. So this
works exactly as in the proof of Theorem~\ref{theorem:main-compact} except that
we also have to invoke
Lemma~\ref{lemma:quasi-isomorphism-deformation-functor-augmented-infinity-MHD}
and Lemma~\ref{lemma:quasi-isomorphism-MHD-L-infinity} since the
controlling mixed Hodge diagram is defined only up to quasi-isomorphism.

We also have to check the functoriality.
A morphism
\begin{equation}
f:(X_1,x_1,\rho_1) \longrightarrow (X_2,x_2,\rho_2) ,
\end{equation}
meaning that we require $f(x_1)=x_2$ and the commutativity of
\begin{equation}
 \xymatrix{
\pi_1(X_1,x_1) \ar[rr]^{f_*} \ar[rd]_{\rho_1} & & \ar[ld]^{\rho_2} \pi_1(X_2, x_2) \\
& G(\kk), & } 
\end{equation}
induces a morphism 
\begin{equation}
\xymatrix{ \Phi_1\curvearrowright \hspace{-7ex} & Y_1 \ar[r]^h \ar[d]_{\pi_1}
& Y_2
\ar[d]^{\pi_2} & \hspace{-7ex} \curvearrowleft\Phi_2 \\
& X_1 \ar[r]_f & X_2 & } 
\end{equation}
compatibly with the base points and equivariant with respect to $\Phi_1$, that
acts on $Y_1$ and on $Y_2$ via the induced morphism of groups
$\phi:\Phi_1 \rightarrow \Phi_2$.
Then we can upgrade this to a morphism of equivariant compactifications: start
with such compactifications $\overline{Y}_1$, $\overline{Y}_2$
and consider the graph $\Gamma_h$ of $h$ seen as a subset
\begin{equation}
\Gamma_h \subset Y_1 \times Y_2 \subset \overline{Y}_1 \times \overline{Y}_2
. 
\end{equation}
By construction, $\Gamma_h$ is $\Phi_1$-invariant, where $\Phi_1$ acts
diagonally. So is its Zariski closure. This defines a morphism
$\overline{Y}_1\rightarrow \overline{Y_2}$ which is $\phi$-equivariant. Such a
morphism induces a morphism of mixed Hodge diagrams
\begin{equation}
 \MHD(\overline{Y_2}, D_{2})_\kk \longrightarrow \MHD(\overline{Y_1},
D_{1})_\kk 
\end{equation}
and then a morphism of mixed Hodge diagrams of Lie algebras
$M_2 \rightarrow M_1$ augmented over $\mathfrak{g}$. So the canonically induced
morphism $\Ohat_{\rho_1} \longrightarrow \Ohat_{\rho_2}$ is a morphism of mixed
Hodge structures.
\end{proof}

As in the compact case, one can also re-write this easily for complex mixed
Hodge structures.

\begin{proposition}
Assume that $G$ and $\rho$ are defined over $\CC$. Then $\Ohat_\rho$ has a
functorial complex mixed Hodge structure.
\end{proposition}

\begin{proof}
In this case $\mathfrak{g}$ is defined over $\CC$. In the mixed Hodge diagram of
Navarro Aznar, keeping only the part over $\CC$ defines a complex mixed Hodge
diagram $\MHD(\overline{Y},D)_\CC$.  The controlling mixed Hodge diagram of Lie
algebras is
\begin{equation}
M := \big( \MHD(\overline{Y}, D)_\CC \otimes \mathfrak{g} \big)^\Phi
\end{equation}
where the tensor product is over $\CC$.
\end{proof}

As announced, the construction of Navarro Aznar allows us to show that in the
compact case the mixed Hodge structure on $\Ohat_\rho$ is defined over $\kk$ if
$\rho$ is. Recall that in this case
$\rho$ is the monodromy of a variation of Hodge structure defined over $\kk$, so
that $V$ is a local system of finite-dimensional vector spaces over $\kk$.

\begin{proof}[{Proof of Proposition~\ref{proposition:main-compact-over-k}}]
Let $\mathcal{L}$ be the sheaf of sections of $\End(V)$. It is a local system of
finite-dimensional Lie algebras over $\kk$. Then $R_{\TW}\Gamma(X,\mathcal{L})$
is a DG Lie algebra over $\kk$ (the construction described
in \cite[\S~3]{Navarro} for commutative algebras works as well for Lie
algebras). So we use it as the part over $\kk$ of the mixed Hodge diagram of Lie
algebras $L$ of Theorem~\ref{theorem:compact-L}. The augmentation at $x$ comes
from the canonical morphism of sheaves
$\mathcal{L} \longrightarrow \mathcal{L}_x $
where the right-hand side is the sheaf supported at $x$ with stalk the Lie
algebra $\End(V_x)$, and with
$R_{\TW}\Gamma(X, \mathcal{L}_x)=\End(V_x)$.
\end{proof}

\section{Description of the mixed Hodge structure}
\label{section:MHS-desctiption}

In this final section we extract a description of the mixed Hodge structure we
constructed on $\Ohat_\rho$. We use only algebraic methods, without referring to
the geometric origin of our mixed Hodge structures. Hence we relax the notations
and write $X,G,\rho$ in all cases. We denote by $L$ a controlling mixed Hodge
diagram of Lie algebras, $\epsilon_x$ the augmentation at the base point $x$ and
$C$ the desuspended mapping cone of $\epsilon_x$.

The first observation that is straightforward is that the mixed Hodge structure
on $\Ohat_\rho$ has only non-positive weights. Namely during the construction of
$L$ and $C$ we always work with non-negative weights, and the induced weights on
$H^0(\mathscr{C}(C))$ are non-negative, but then we dualize and this produces
the non-positive weights.

The second easy observation is that we get an explicit description of the mixed
Hodge structure on the cotangent space to $\Ohat\rho$.

\begin{theorem}
The mixed Hodge structure on the cotangent space to $\Ohat_\rho$ is dual to the
mixed Hodge structure on $H^1(C)$, which fits into a short exact sequence
\begin{equation}
\label{equation:MHS-cotangent-space}
0 \longrightarrow \mathfrak{g}/\epsilon_x(H^0(L)) \longrightarrow H^1(C)
\longrightarrow H^1(L) \longrightarrow 0 .
\end{equation}
\end{theorem}

\begin{proof}
The coalgebra $H^0(\mathscr{C}(C))$ has a canonical filtration which is dual to
the sequence of quotients of $\Ohat_\rho$ by powers of the maximal ideal. In
particular its first term $H^0(\mathscr{C}_1(C))$ is dual to the cotangent
space. But by construction $\mathscr{C}_1(C)$ is simply $C[1]$. The long exact
sequence for $C$ yields exactly the exact
sequence~\eqref{equation:MHS-cotangent-space} with $H^0(C[1])=H^1(C)$.
\end{proof}

In our situations this leads us to the following descriptions:
\begin{enumerate}
\item If $X$ is compact then $H^1(L)$ is pure of weight $1$. So $H^1(C)$ has
this as weight $1$ part and a weight $0$ part coming from
$\mathfrak{g}/\epsilon_x(H^0(L))$. We see that our mixed Hodge structure
coincides with the one of Eyssidieux-Simpson on the cotangent space.
\item If $X$ is smooth quasi-projective then $H^1(L)$ has weights only
$1,2$. Furthermore if $\rho$ has finite image then $\epsilon_x$ is surjective
and the left-hand side vanishes. So the only weights are $1,2$. Splitting the
weight filtration on the cotangent space and lifting a basis to generators of
$\Ohat_\rho$ produces generators of weight $1,2$ as in the theorem of
Kapovich-Millson.
\item As a particular case, if $H^1(L)$ is pure of weight $2$ (which happens if
the mixed Hodge structure on the $H^1$ of the finite cover corresponding to
$\rho$ is pure of weight $2$) then $\Ohat_\rho$ has only homogeneous
generators. In this sense one recovers the main theorem of \cite{Lefevre}.
\item In the general non-compact case where $\rho$ is the monodromy of a
variation of Hodge structure we expect generators of weight $0,1,2$ on
$\Ohat_\rho$.
\end{enumerate}

In order to recover completely the theorem of Kapovich-Millson, we would like to
get a presentation of $\Ohat_\rho$ as a quotient of the formal power series
algebra on $H^1(C)$ by an ideal carrying a mixed Hodge structure with weights
$2,3,4$ coming from $H^2(C)=H^2(L)$. We are unable to get this.

Now we identify the weight filtration of $\Ohat_\rho$ globally. The group $G$
acts algebraically on $\Hom(\pi_1(X,x),G)$ and the orbit of $\rho$ defines a
reduced closed subscheme $\Omega_\rho$. Formally at $\rho$, $\Omega_\rho$ is
defined by an ideal $\mathfrak{j}\subset\Ohat_\rho$. The quotient
$\Ohat_\rho/\mathfrak{j}$ is the algebra of formal function on $\Omega_\rho$ at
$\rho$.

\begin{theorem}
The weight zero part of $\Ohat_\rho$ is the formal local ring of the orbit
of~$\rho$.
\end{theorem}

At the level of the cotangent space, we see this from the previous theorem.

\begin{proof}
We will write an explicit morphism from $\Ohat_\rho$ to the formal local ring of
$\Omega_\rho$ using our methods of cones, functor $\mathscr{C}$ and abstract
pro-representability theorems. The kernel of this morphism will be the ideal
$\mathfrak{j}$ and the compatibility of this construction with mixed Hodge
structures will be clear.

Let us consider the desuspended mapping cone $D$ of
$ \epsilon_x : H^0(L) \rightarrow \mathfrak{g} . $
This has $D^0=H^0(L)$, $D^1=\mathfrak{g}$, with the differential being given by
$\epsilon_x$. Its $L_\infty$ algebra structure is in fact DG Lie since the only
non-zero Lie brackets are the one induced on $D^0$ and the naive bracket between
$D^0$ and $D^1$. The deformation functor is given by ($A\in \Art_\kk$)
\begin{equation}
\Def_D(A) = \exp(\mathfrak{g}\otimes\mathfrak{m}_A) / \exp(H^0(L)\otimes\mathfrak{m}_A).
\end{equation}

Let us consider also the DG Lie algebra $E$ which has only
$E^1=\mathfrak{g}/\epsilon_x(H^0(L))$ with zero bracket and differential. One
can see $E$ as the desuspended mapping cone of
$0\rightarrow \mathfrak{g}/\epsilon_x(H^0(L))$, the $L_\infty$ algebra structure
being trivial in this case. The deformation functor is given by
\begin{equation}
\Def_E(A) = (\mathfrak{g}/\epsilon_x(H^0(L)))\otimes\mathfrak{m}_A .
\end{equation}

There is a canonical quasi-isomorphism $D\stackrel{\approx}{\longrightarrow} E$.
One can see this as being induced by a morphism between the morphisms whose
$D,E$ are the cones, hence by the functoriality of the $L_\infty$ (in fact DG
Lie) structure on the mapping cone this is morphism of DG Lie algebras, and it
induces an isomorphism of the deformation functors. Hence there is a canonical
isomorphism
\begin{equation}
H^0(\mathscr{C}(D))^*=H^0(\mathscr{C}(E))^* .
\end{equation}
But the computation of $H^0(\mathscr{C}(E))^*$ is straightforward: it is (up to
adding $\kk$) the
algebra of power series on the vector space
$\mathfrak{g}/\epsilon_x(H^0(L))$.

Now we follow \cite[\S~2.2.2]{EyssidieuxSimpson} where it is explained how this
corresponds to the algebra of formal functions on the orbit of $\rho$. Observe
that, following their notations, our $D$ corresponds to their deformation
functor $h'_1$ and $E$ to $h_1$. Then there is a natural morphism
$D\rightarrow C$ which corresponds to the morphism of deformation functors
$h'_1 \rightarrow \Def_\rho$ (the inclusion of the orbit inside the
representation variety). It induces a morphism
\begin{equation}
\label{equation:description-MHS-orbit-projection}
p:\kk \oplus H^0(\mathscr{C}(C))^* \longrightarrow \kk\oplus H^0(\mathscr{C}(D))^* .
\end{equation}
Via their pro-representability property, the left-hand side is identified with
$\Ohat_\rho$ and the right-hand side is identified with the algebra of formal
functions of $\Omega_\rho$. So the kernel of $p$ is the ideal $\mathfrak{j}$
defining the orbit inside $\Ohat_\rho$.

If we replace in the right-hand side
of~\eqref{equation:description-MHS-orbit-projection} $E$ by $D$ we
see clearly that the mixed Hodge structure is pure of weight zero.  By the
strict compatibility of the weight filtration with respect to morphisms of mixed
Hodge structures we get that $\Gr^W_0(\Ohat_\rho)$ equals this right-hand side.
\end{proof}

\bibliographystyle{amsalpha}
\bibliography{Bibliographie}

\end{document}